\documentclass[12pt]{article}%
\usepackage{amsfonts}
\usepackage{amsmath}
\usepackage{amssymb}
\usepackage{graphicx}%
\providecommand{\U}[1]{\protect\rule{.1in}{.1in}}
\newtheorem{theorem}{Theorem}

\newtheorem{corollary}[theorem]{Corollary}

\newtheorem{lemma}[theorem]{Lemma}

\newtheorem{proposition}[theorem]{Proposition}

\newenvironment{proof}[1][Proof]{\noindent\textbf{#1.} }{\ \rule{0.5em}{0.5em}}
\begin{document}

\title{The spherical Whittaker Inversion Theorem and the quantum non-periodic Toda Lattice}
\author{Nolan R. Wallach}
\maketitle

\begin{abstract}
In this paper the spherical case of the Whittaker Inversion Theorem is given a
relatively self-contained proof. This special case can be used as a help in
deciphering the handling of the continuous spectrum in the proof of the full
theorem. It also leads directly to the solution of the quantum non-periodic
Toda Lattice. This is also explained in detail in this paper.

\end{abstract}

\section{Introduction}

The main purpose of this paper is to give a complete, relatively self
contained, proof of the spherical Whittaker Inversion Theorem for real
reductive groups. This result is an important special case of the general
theorem but it is unencumbered by the complications caused by discrete
spectrum. Reading it can be used as a help in understanding the arguments used
in \cite{Whittaker-Plancherel} to handle the continuous spectrum. By
relatively self contained I mean that it will be based on two main results:
The first is the Harish-Chandra Plancherel Formula for $L^{2}(G/K)$ (actually
the inversion formula) as developed in the work of Helgason \cite{Helgason1}%
,\cite{Hellgason2}.The second is the holomorphic continuation of the Jacquet
Integral for minimal parabolic subgroups and its brilliant implication due to
Rapha\"{e}l Beauzart-Plessis \cite{raphael}. The full theorem was announced in
the 1980's (see \cite{RRGII}) but a correct proof has only recently appeared
in \cite{Whittaker-Plancherel}. The applications of these results to the
theory of automorphic forms inundate the literature. We will include in the
last section of this paper this paper an application to the quantum
non-periodic Toda lattice (see \cite{Goodman-Wallach} for more background on
the subject). If you are a physicist or a mathematician who is not an expert
in Representation Theory, I would recommend that you read Section
\ref{TodaChapter} of this paper first.

\section{Some notation}

Let $G$ be a real reductive group. In this paper there is no loss of
generality to take $G$ to be the identity component of a subgroup of
$GL(n,\mathbb{R})$ that is the locus of zeros of a set of polynomials on
$M_{n}(\mathbb{R})$ such that $G$ invariant under transpose ($g\mapsto g^{T}%
$). We set $K=G\cap O(n)$ a maximal compact subgroup of $G$. We choose an
Iwasawa decomposition of $G$ given by $G=NAK$ with $N$ a maximal unipotent
subgroup (i.e. the elements of $N$ are of the form $I+X$ with $X$ nilpotent)
of $G$, $A$ a subgroup maximal among the subgroups of $G$ contained in the set
of symmetric positive definite matrices and $aNa^{-1}\subset N,a\in A$. Let
$\mathfrak{a}=Lie(A),\mathfrak{n}=Lie(N)$ and $\mathfrak{g}=Lie(G)$. Then
$ad(h)_{|\mathfrak{n}}$ for $h\in\mathfrak{a}$ simultaneously diagonalize
yielding the roots $\Phi^{+}$ of $\mathfrak{a}$ on $\mathfrak{n}$. The
exponential map $\exp:\mathfrak{a}\rightarrow A$ is a Lie group isomorphism
($\mathfrak{a}$ is a group under addition) set $\log:A\rightarrow\mathfrak{a}$
equal to the inverse map. We define for $\lambda\in\mathfrak{a}_{\mathbb{C}%
}^{\ast},a\in A$. $a^{\lambda}=e^{\lambda(\log(a))}.$ Define%
\[
\rho(h)=\frac{1}{2}\mathrm{tr}(ad(h)_{|\mathfrak{n}}).
\]
If $x,y\in\mathfrak{g}$ then set $\left\langle x,y\right\rangle =\mathrm{tr}%
(xy^{T})$ and $\left\Vert x\right\Vert =\left\langle x,x\right\rangle
^{\frac{1}{2}}$. Let $m=\dim\mathfrak{n}.$ On $\wedge^{m}\mathfrak{g}$ we put
the inner product induced by $\left\langle ...,...\right\rangle $. Let
$Int(g)$ be the transformation of $\,\mathfrak{g}$ given by $Int(g)x=gxg^{-1}%
$. Define for $g\in$ $G$, $\left\vert g\right\vert $ to be the operator norm
of $\wedge^{m}Int(g)$. Set $A_{G}$ equal to the subgroup of elements of the
center of $G$ that are positive definite. Then $KA_{G}[G,G]=G$. Define
\[
\left\Vert kag\right\Vert =e^{\left\Vert \log a\right\Vert }\left\vert
g\right\vert ^{\frac{1}{2}},k\in K.a\in A_{G},g\in\lbrack G,G].
\]
If $g\in G$ then $g$ can be written in the form $g=k_{1}ak_{2}$ with $a\in A$
such such that $a^{\alpha}\geq1$ for $\alpha\in\Phi^{+}$ and $k_{1},k_{2}\in
K$. One can check that if $g$ is of this form and if $g\in\lbrack G,G]$ then%
\[
\left\vert g\right\vert =a^{2\rho}\text{. }%
\]
It is easily seen that%
\[
\left\Vert x\right\Vert \geq1,\left\Vert xy\right\Vert \leq\left\Vert
x\right\Vert \left\Vert y\right\Vert
\]
and that the sets%
\[
\left\Vert x\right\Vert \leq r
\]
are compact for $r<\infty$. It is a bit harder to prove that there exists $d$
such that
\[
\int_{G}\left\vert x\right\vert ^{-1}(1+\log\left\Vert x\right\Vert
)^{-d}dx<\infty.
\]

Let $M=\{k\in K|ka=ak,a\in A\}$.

We now recall the Harish-Chandra Schwartz space. If $f\in C^{\infty}(G)$,
$x,y\in U(\mathfrak{g})$ $d\in\mathbb{R}$ then set
\[
p_{x,y,d}(f)=\sup\{\left\vert R_{x}L_{y}f(g)\right\vert \left\vert
g\right\vert ^{\frac{1}{2}}(1+\log\left\Vert g\right\Vert )^{d}|g\in G\}.
\]
Here if $X\in U(\mathfrak{g})$ then $R_{X}f(g)=\frac{d}{dt}f(g\exp tX)_{|t=0}$
and $L_{Y}f(g)=\frac{d}{dt}f(\exp(-tX)g)_{|t=0}$ and since $[L_{X}%
,L_{Y}]=L_{[X,Y]}$ and $[R_{X},R_{Y}]=R_{[X,Y]}$ the universal mapping
property of the universal enveloping algebra allows us to define $L_{x}$ and
$L_{y}$ for all $x.y\in U(\mathfrak{g}).$

Define $\mathcal{C}(G)=\{f\in C^{\infty}(G)|p_{x,y,d}(f)<\infty,x\in
U(\mathfrak{g}),d\in\mathbb{R\}}$ endowed with the topology defined by the
semi-norms $p_{x,y,d}$.

\section{The Plancherel Theorem for $\mathcal{C}(G/K)$}

Note that the map $\theta:G\rightarrow G$ given by $\theta(g)=(g^{T})^{-1}$ is
an automorphism of $G$. Set $\bar{N}=\theta(N)$. We have the corresponding
Iwasawa decomposition $G=\bar{N}AK$. Define for $g=\bar{n}ak$, $\bar{n}\in
\bar{N},a\in A,k\in K,$ $a(g)=a,k(g)=k$. Since the Iwasawa decomposition is
unique, $a:G\rightarrow A$ and $k:G\rightarrow K$ are $C^{\infty}$. If $f\in
C^{\infty}(M\backslash K)$ define for $\nu\in\mathfrak{a}_{\mathbb{C}}^{\ast
}$
\[
f_{\nu}(\bar{n}ak)=a^{\nu-\rho}f(k).
\]
This defines $f_{\nu}$ as a $C^{\infty}$ function on $G$. Define an action of
$G$ on $C^{\infty}(M\backslash K)$ by%
\[
\left(  \pi_{\nu}(g)f\right)  (k)=f_{\nu}(kg)\text{.}%
\]
If we put the $C^{\infty}$ topology on $C^{\infty}(M\backslash K)$ then
$(\pi_{\nu},C^{\infty}(M\backslash K))$ is a smooth Fr\'{e}chet representation
of $G$. We also put the $L^{2}$--inner product on $C^{\infty}(M\backslash K)$
\[
\left\langle u,w\right\rangle =\int_{K}u(k)\overline{w(k)}dk
\]
and find that
\[
\left\langle \pi_{\nu}(g)u,w\right\rangle =\left\langle w,\pi_{-\bar{\nu}%
}(g^{-1})u\right\rangle
\]
where $\bar{\nu}(h)=\overline{\nu(h)}$. In particular, if $\nu\in
\mathfrak{a}^{\ast}$ then $(\pi_{i\nu},L^{2}(M\backslash K))$ is a unitary
representation. Let $f\in C_{c}^{\infty}(G/K)$ then set%
\[
\left(  \pi_{\nu}(f)u\right)  (k)=\int_{G}u_{\nu}(kg)f(g)dg.
\]

If $\lambda\in\mathfrak{a}^{\ast}$ define $H_{\lambda}\in\mathfrak{a}^{\ast}$
by $\left\langle H_{\lambda},h\right\rangle =\lambda(h)$ for $h\in
\mathfrak{a}$. Define $(\lambda,\mu)=\left\langle H_{\lambda},H_{\mu
}\right\rangle $ and extend $(...,...)$ to $\mathfrak{a}_{\mathbb{C}}^{\ast}$
bilinearly. Then set%
\[
c(\nu)=\int_{N}a(n)^{\nu-\rho}dn
\]
where we leave, for the moment, the bi-invariant measure on $N$ unnormalized.
This integral converges absolutely for $\nu\in\mathfrak{a}_{\mathbb{C}}^{\ast
}$ such that%
\[
\operatorname{Re}(\nu,\alpha)<0
\]
for all $\alpha\in\Phi^{+}$ and uniformly in compacta in this set. Thus
$c(\nu)$ defines a holomorphic function on this subset. It has a meromorphic
continuation to all of $\mathfrak{a}_{\mathbb{C}}^{\ast}$ indeed
Gindikin-Karpelovic derived an explicit formula (c.f. \cite{Hellgason2} or
\cite{HArmHom}) 8.10.18) if we set for $\alpha\in\Phi_{0}^{+}=\{\beta\in
\Phi^{+}|\frac{\beta}{2}\notin\Phi^{+}\}$,
\[
c_{\alpha}(\nu)=\left\{
\begin{array}
[c]{c}%
B(\frac{\dim\mathfrak{n}_{\alpha}}{2},\frac{(\nu,\alpha)}{(\alpha,\alpha
)})\text{ if }2\alpha\notin\Phi^{+}\\
B(\frac{\dim\mathfrak{n}_{\alpha}}{2},\frac{(\nu,\alpha)}{(\alpha,\alpha
)})B(\frac{\dim\mathfrak{n}_{2\alpha}}{2},\frac{(\nu,\alpha)}{2(\alpha
,\alpha)}+\frac{\dim\mathfrak{n}_{\alpha}+\dim\mathfrak{n}_{2\alpha}}%
{2})\text{ if }2\alpha\in\Phi^{+}%
\end{array}
\right.  .
\]
Then the the measure on $N$ can be normalized so that
\[
c(\nu)=\prod_{\alpha\in\Phi_{0}^{+}}c_{\alpha}(\nu).
\]
Set for $\nu\in\mathfrak{a}^{\ast}$
\[
\mu(\nu)=\frac{1}{c(i\nu)c(-i\nu)}=\frac{1}{\left\vert c(i\nu)\right\vert
^{2}}%
\]
one can show using the formula for $c(\nu)$ and basic properties of the
$\Gamma$--function that that $\mu(\nu)\leq C(1+\left\Vert \nu\right\Vert
^{r})$ for some $r$.

We are now ready develop the Plancherel theorem for $G/K.$ We first recall a
special case of the Harish-Chandra Plancherel Theorem:

\begin{theorem}
The measure on $G$ and $\mathfrak{a}^{\ast}$ can be normalized so that if
$\phi\in\mathcal{C}(K\backslash G/K)$ (that is $\phi(k_{1}gk_{2}%
)=f(g),k_{1},k_{2}\in K$) then
\[
\phi(g)=\int_{\mathfrak{a}^{\ast}}\left\langle \pi_{i\nu}(L_{g^{-1}}%
\phi)1,1\right\rangle \mu(\nu)d\nu.
\]

\end{theorem}

There is a relatively elementary proof of this Theorem due to Anker
\cite{Anker} that specifically proves this result. Harish-Chandra proved much
more. We recall an argument in Helgason \cite{Hellgason2} section III.1, in
the proof of the following implication.

\begin{theorem}
Let $f\in\mathcal{C}(G/K)$ then with the same normalizations as in the
previous result we have%
\[
f(g)=\int_{\mathfrak{a}^{\ast}}\left\langle \pi_{i\nu}(L_{g^{-1}%
}f)1,1\right\rangle \mu(\nu)d\nu.
\]

\end{theorem}

\textbf{Note:}The point here is that $f$ is not necessarily left $K$--finite
(since under that condition the result will be a special case of
Harish-Chandra's Theorem).

\begin{proof}
Define for $g,x\in G$
\[
\phi(g,x)=\int_{K}f(gkx)dk.
\]
Noting that the Casimir operator, $C$, corresponding to the choice of $B$ on
$\mathfrak{g}$ yields the Laplacian of the Riemannian structure on $G$ given
by the inner product on $\mathfrak{g}/Lie(K)$ induced by $B$. This and Sobolev
theory implies that if $g$ is fixed then $x\mapsto\phi(g,x)$ is in
$\mathcal{C}(K\backslash G/K)$. Thus keeping $g$ fixed we have%
\[
\phi(g,I)=\int_{\mathfrak{a}^{\ast}}\left\langle \pi_{i\nu}(\phi
(g,\cdot))1,1\right\rangle \mu(\nu)d\nu.
\]
Now
\[
\left\langle \pi_{i\nu}(\phi(g,\cdot))1,1\right\rangle =\int_{G}%
\phi(g,x)\left\langle \pi_{i\nu}(x)1,1\right\rangle dx=\int_{G}\int
_{K}f(gkx)\left\langle \pi_{i\nu}(x)1,1\right\rangle dkdx
\]%
\[
=\int_{G}\int_{K}f(gx)\left\langle \pi_{i\nu}(k^{-1}x)1,1\right\rangle
dkdx=\int_{G}f(gx)\left\langle \pi_{i\nu}(x)1,1\right\rangle dx=\left\langle
\pi_{i\nu}(L_{g^{-1}}f)1,1\right\rangle .
\]
Noting that $\phi(g,I)=f(g)$ completes the proof.
\end{proof}

Harish-Chandra's theorem says much more in the $K$-finite case.

\begin{theorem}
\label{HCVersion}Let $u\in C^{\infty}(M\backslash K)$ be right $K$--finite and
let $\alpha\in\mathcal{S}(\mathfrak{a}^{\ast})$ then the function $f$ defined
by%
\[
f(g)=\int_{\mathfrak{a}^{\ast}}\left\langle \pi_{i\nu}(g)1,u\right\rangle
\alpha(\nu)\mu(\nu)d\nu
\]
is in $\mathcal{C}(K\backslash G).$
\end{theorem}

\section{The holomorphic continuation of Jacquet integrals and a Theorem of
Beuzart-Plessis}

Retain the notation of the previous sections. If $\chi:N\rightarrow S^{1}$ is
a unitary (one dimensional) character of $N$ we consider for $u\in C^{\infty
}(M\backslash K)$
\[
J_{\chi,\nu}(u)=\int_{N}\chi(n)^{-1}u_{\nu}(n)dn.
\]
Note that if $\chi=1$ and $u=1$ then the integral is the one that defines the
Harish-Chandra c-function. This implies that the integral defining
$J_{\chi,\nu}$ converges absolutely if $\operatorname{Re}(\nu,\alpha)<0$ for
all $\alpha\in\Phi^{+}$. Let $\Delta$ be the set of elements of $\Phi^{+}$
that appear in $\mathfrak{n}/[\mathfrak{n},\mathfrak{n}].$ Thus $\mathfrak{n}%
/[\mathfrak{n},\mathfrak{n}]=\oplus_{\alpha\in\Delta}\left(  \mathfrak{n}%
/[\mathfrak{n},\mathfrak{n}]\right)  _{\alpha}$. We say that $\chi$ is generic
if its differential is non-zero on each of the spaces $\left(  \mathfrak{n}%
/[\mathfrak{n},\mathfrak{n}]\right)  _{\alpha}$ with $\alpha\in\Delta$. \ The
Jacquet integrals are the $J_{\chi,\nu}$ with $\chi$ generic. The holomorphic
continuation of Jacquet integrals in this generality (non-K-finite $u$) was
first proved in \cite{JacquetInt} (c.f. \cite{RRGII} Theorem 15.6.7).

\begin{theorem}
If $\chi$ is generic and if $u\in C^{\infty}(M\backslash K)$ then $J_{\chi
,\nu}(u)$ has a holomorphic continuation to $\mathfrak{a}_{\mathbb{C}}^{\ast}%
$, furthermore the map $\nu\mapsto J_{\chi,\nu}$ is a weakly holomorphic map
of $\mathfrak{a}_{\mathbb{C}}^{\ast}$ to $C^{\infty}(M\backslash K)^{\prime}$
(the continuous dual space). Finally, if $\nu\in\mathfrak{a}_{\mathbb{C}%
}^{\ast}$ then $J_{\chi,\nu}\neq0$.
\end{theorem}

If $\chi$ is not generic then one can use this result combined with the
meromorphic continuation of conical vectors to prove a meromorphic
continuation result.

The Theorem of Beuzart-Plessis \cite{raphael} which is based on the holomorphy
of $\nu\mapsto J_{\chi,\nu}(1)$ is

\begin{theorem}
\label{Raphael}If $\chi$ is a generic character of $N$ then there exists
$\varepsilon>0$ such that
\[
\int_{\mathrm{\ker}(\chi)}a(n)^{-(1-\varepsilon)\rho}dn<\infty.
\]

\end{theorem}

We will also need the following results (see \cite{Whittaker-Plancherel}
Theorem 43 for the full details of a proof of Proposition \ref{tempered})
whose proof is complicated and uses parts of the proof of the holomorphic
continuation of the Jacquet integrals (\cite{JacquetInt}). We will just give
an idea of why they are true.

\begin{lemma}
Assume that $\chi$ \ is generic. There exists a continuous semi-norm, $q$, on
$C^{\infty}(M\backslash K)$ such that
\[
\left\vert J_{\chi,\nu}(u)\right\vert \leq q(u)c(\operatorname{Re}\nu)
\]
if $u\in C^{\infty}(M\backslash K)$ and $\operatorname{Re}(\nu,\alpha)<0$ for
$\alpha\in\Phi^{+}$.
\end{lemma}

\begin{proof}
We have if $\nu$ satisfies the condition then%
\[
J_{\chi,\nu}(u)=\int_{N}\chi(n)^{-1}u_{\nu}(n)dn=\int_{N}\chi(n)^{-1}%
a(n)^{\nu-\rho}u(k(n))dn
\]
so defining $q(u)=\sup_{k\in K}\left\vert u(k)\right\vert $ we have%
\[
\left\vert J_{\chi,\nu}(u)\right\vert \leq q(u)\int_{N}a(n)^{\operatorname{Re}%
\nu-\rho}d\nu=q(u)c(\operatorname{Re}\nu)\text{.}
\]

\end{proof}

\begin{proposition}
\label{tempered}Assume that $\chi$ is generic. And let $0<r<\infty.$ There
exists a continuous semi-norm, $q_{r}$, on $C^{\infty}(M\backslash K)$ and
$m_{r}$ such that%
\[
\left\vert J_{\chi,\nu}(u)\right\vert \leq q_{r}(u)(1+\left\Vert
\operatorname{Im}\nu\right\Vert )^{m_{r}}%
\]
for $\nu\in\mathfrak{a}_{\mathbb{C}}^{\ast}$ such that $0>\operatorname{Re}%
(\nu,\alpha)>-r(\rho,\alpha),\alpha\in\Phi^{+}$.
\end{proposition}

\begin{proof}
This is proved using an argument involving tensoring with finite dimensional
representations and details from the shift argument used in the proof of the
holomorphic continuation of $J_{\chi,\nu}(u)$ in Section 15.5 of \cite{RRGII}.
\end{proof}

\begin{proposition}
\label{Key-Estimate}Assume that $\chi$ is generic. If $0<r<\infty$ is fixed
and if $f\in\mathcal{C}(G/K)$ then for each $m$ there exists $C_{l,r} $ such
that%
\[
|J_{\chi,i\nu-z\rho}(\pi_{i\nu}(f)1)|\leq C_{m,r}(1+\left\Vert \nu\right\Vert
)^{-m}
\]
for $\nu\in\mathfrak{a}^{\ast}$ and $0\geq\operatorname{Re}z>-r$
\end{proposition}

\begin{proof}
We note that if $p$ is a continuous semi-norm on $C^{\infty}(M\backslash K)$
then there exists $k$ and a constant $L$ such that%
\[
p(u)\leq L_{P}\left\Vert (1+C_{K})^{k_{P}}u\right\Vert .
\]
Thus Proposition \ref{tempered} implies that
\[
|J_{\chi,i\nu-z\rho}(\pi_{i\nu}(f)1)|\leq q_{r}(\pi_{i\nu}(f)1)(1+\left\Vert
\nu\right\Vert )^{m_{r}}%
\]%
\[
\leq L_{q_{r}}\left\Vert \pi_{i\nu}(L_{(1+C_{K})^{k_{q_{r}}}}f)1\right\Vert
(1+\left\Vert \nu\right\Vert )^{m_{r}}.
\]
Now applying Lemma \ref{Schwartz-estimate} in Appendix 1 to $L_{(1+C_{K}%
)^{k_{q_{r}}}}f$ with $m=m_{r}+l$ completes the proof.
\end{proof}

By analogy with the Harish-Chandra Schwartz space we have the Whittaker
Schwartz space. Which we now recall. If $g\in G$ and $g=nak$ $n\in N,a\in
A,k\in K$ then set $a_{o}(g)=a$. That is $a_{o}(g)=a(\theta(g))^{-1}$.
\[
C^{\infty}(N\backslash G;\chi)=\{f\in C^{\infty}(G)|f(ng)=\chi(n)f(g),n\in
N,g\}.
\]
If $f\in C^{\infty}(N\backslash G;\chi),x\in U(\mathfrak{g})$ then set
\[
q_{x,d}(f)=\sup_{g\in G}a_{o}(g)^{-\rho}(1+\left\Vert \log a_{o}(g)\right\Vert
)^{d}|R_{x}f(g)|.
\]
Then $\mathcal{C}(N\backslash G;\chi)$ is the space of $f\in C^{\infty
}(N\backslash G;\chi)$ such that $q_{x,d}(f)<\infty$ for all $x,d$.

The following observation will be used in the last section.

\begin{lemma}
\label{asymptotic}If $H\in\mathfrak{a}$ and $\alpha(H)<0$ for all $\alpha
\in\Phi^{+}$ and if $\nu\in\mathfrak{a}_{\mathbb{C}}^{\ast}$ is such that
$\operatorname{Re}(\nu,\alpha)<0$ for all $\alpha\in$ $\Phi^{+}$ then%
\[
\lim_{t\rightarrow+\infty}e^{-t(\nu+\rho)}J_{\chi,\nu}(\pi_{\nu}(\exp
tH)1)=c(\nu)\text{.}%
\]

\end{lemma}

\begin{proof}
We calculate%
\[
\int_{N}\chi(n)^{-1}a(n\exp tH)^{\nu-\rho}dn=\int_{N}\chi(n)^{-1}a(\exp
tH\exp-tHn\exp tH)^{\nu-\rho}dn
\]%
\[
=e^{t(\nu-\rho)(H)}\int_{N}\chi(n)^{-1}a(\exp-tHn\exp tH)^{\nu-\rho
}dn=e^{t(\nu+\rho)(H)}\int_{N}\chi(\exp tHn\exp-tH)^{-1}a(n)^{\nu-\rho}dn.
\]
Now multiply by $e^{-t(\nu+\rho)}$ and take the limit.
\end{proof}

\section{The key formula}

In this section $f\in\mathcal{C}(G/K)$. Then we have seen that%
\[
f(g)=\int_{\mathfrak{a}^{\ast}}\left\langle \pi_{i\nu}(L_{g^{-1}%
}f)1,1\right\rangle \mu(\nu)d\nu=\int_{\mathfrak{a}^{\ast}}\left\langle
\pi_{i\nu}(f)1,\pi_{i\nu}(g)1\right\rangle \mu(\nu)d\nu.
\]
Thus
\[
\overline{f(g)}=\int_{\mathfrak{a}^{\ast}}\left\langle \pi_{i\nu}%
(g)1,\pi_{i\nu}(f)1\right\rangle \mu(\nu)d\nu.
\]

\begin{theorem}
\label{KeyFormula}$\int_{N}\chi(n)\overline{f(ng)}dn=\int_{\mathfrak{a}^{\ast
}}J_{\chi^{-1},i\nu}(\pi_{i\nu}(g)1)\overline{J_{\chi^{-1},i\nu}(\pi_{i\nu
}(f)1)}\mu(\nu)d\nu.$
\end{theorem}

\begin{proof}
In this proof we will use the notation $u(\nu)=\pi_{i\nu}(f)1\in C^{\infty
}(M\backslash K)$. Using the non-compact model for the unitary principal
series (c.f. \cite{HArmHom} 8.4.7) one has
\[
\left\langle \pi_{i\nu}(g)1,\pi_{i\nu}(f)1\right\rangle =\int_{N}%
a(ng)^{i\nu-\rho}a(n)^{-i\nu-\rho}\overline{u(\nu)(k(n))}dn.
\]
Thus%
\[
\int_{N}\chi(n)\overline{f(ng)}dn=\int_{N}\chi(n_{1})\int_{\mathfrak{a}^{\ast
}}\int_{N}a(nn_{1}g)^{i\nu-\rho}a(n)^{-i\nu-\rho}\overline{u(\nu)_{i\nu
}(k(n))}dn\mu(\nu)d\nu dn_{1}.
\]
We first deform the parameter and consider%
\[
\int_{N}\chi(n_{1})\int_{\mathfrak{a}^{\ast}}\int_{N}a(nn_{1}g)^{i\nu
-(1+z)\rho}a(n)^{-i\nu-(1+z)\rho}\overline{u(\nu)_{i\nu}(k(n))}dn\mu(\nu)d\nu
dn_{1}%
\]
for $\operatorname{Re}z<0$. If we put absolute values on all of the terms we
have%
\[
\int_{N}\int_{\mathfrak{a}^{\ast}}\int_{N}a(nn_{1}g)^{-(1+\operatorname{Re}%
z)\rho}a(n)^{-(1+\operatorname{Re}z)\rho}\left\vert \overline{u(\nu)_{i\nu
}(k(n))}\right\vert dn\mu(\nu)d\nu dn_{1}%
\]
and using Lemma \ref{Schwartz-estimate} in Appendix 1 (and the notation
therein) we have $\left\vert \overline{u(\nu)_{i\nu}(k(n))}\right\vert \leq
C_{1,l}(1+\left\Vert \nu\right\Vert )^{-l}$ with $C_{1,l}<\infty$ Thus the
integrand is dominated by
\[
\int_{\mathfrak{a}^{\ast}}\int_{N}\int_{N}a(nn_{1}g)^{-(1+\operatorname{Re}%
z)\rho}a(n)^{-(1+\operatorname{Re}z)\rho}dn(1+\left\Vert \nu\right\Vert
)^{-m}d\nu dn_{1}%
\]
since $\mu(\nu)\leq B(1+\left\Vert \nu\right\Vert )^{r}$ for some $r$ and we
take $m$ to be greater than $\dim A$. which converges for $\operatorname{Re}%
z<0$. noting that $a(ng)^{-\rho}\leq C_{\omega}a(n)^{-\rho}$ if $g\in\omega$ a
compact set we see that the integral of the absolute values is dominated by a
multiple of%
\[
C_{\omega}^{(1+\operatorname{Re}z)\rho}\int_{N}\int_{N}a(n_{1}%
)^{-(1+\operatorname{Re}z)\rho}a(n)^{-(1+\operatorname{Re}z)\rho}%
dndn_{1}<\infty.
\]
We can therefore do the deformed integral in any order. We choose%
\[
\int_{N}\chi(n_{1})\chi(n)^{-1}\int_{\mathfrak{a}^{\ast}}\int_{N}%
a(n_{1}g)^{i\nu-(1+z)\rho}a(n)^{-i\nu-(1+z)\rho}\overline{u(\nu)(k(n))}%
dn\mu(\nu)d\nu dn_{1}=
\]%
\[
\int_{\mathfrak{a}^{\ast}}J_{\chi^{-1},i\nu-z\rho}(\pi_{i\nu-z\rho
}(g)1)\overline{J_{\chi^{-1},i\nu-\bar{z}\rho}(\pi_{i\nu}(f)1)}\mu(\nu)d\nu.
\]
We are left with taking the limit under the integral sign $z\rightarrow0$.
This will be done indirectly.

Let $x_{o}$ be perpendicular to $\ker d\chi$ relative to $\left\langle
...,...\right\rangle $ and assume that $d\chi(x_{o})=i1.$ We define%
\[
\tau_{z}(t)=\int_{\ker\chi}\int_{\mathfrak{a}^{\ast}}\int_{N}a(n\exp
tx_{o}n_{1}g)^{i\nu-(1+z)\rho}a(n)^{-i\nu-(1+z)\rho}\overline{u(\nu
)(k(n))}dn\mu(\nu)d\nu dn_{1}%
\]
for Re$z\geq0$. Then Proposition \ref{less-simple-ineq} in the Appendix 1
implies that if $\omega$ is a compact subset of $G$ then
\[
\left\vert \tau_{z}(t)\right\vert \leq C_{\omega}^{(1+\operatorname{Re}%
z)}B(1+\left\vert t\right\vert )^{d}%
\]
with $C_{\omega},d$ and $B$ finite. Fix $g\in G$. The estimate implies that we
can define a family of tempered distribution on $\mathbb{R}$ by%
\[
T_{z}(\phi)=\int_{-\infty}^{\infty}\tau_{z}(t)\phi(t)dt
\]
for $\operatorname{Re}z\geq0$. Dominated convergence implies that the map
$z\mapsto T_{z}$ is a weakly continuous map of $\{z|\operatorname{Re}z\geq0\}$
to $\mathcal{S}^{\prime}(\mathbb{R)}$ (the space of tempered distributions).
Thus if $\mathcal{F}$ is the usual Fourier transform
\[
\mathcal{F}(\phi)(s)=\frac{1}{\sqrt{2\pi}}\int_{-\infty}^{\infty}e^{-ist}%
\phi(t)dt
\]
which is a continuous linear isomorphism of $\mathcal{S}(\mathbb{R}).$ If $T$
is a tempered distribution we define (as usual) $\mathcal{F}(T)=T\circ
\mathcal{F}$. \ If $T$ is given by integration by an element of $L^{1}%
(\mathbb{R})$, i.e. $T(\phi)=\int_{-\infty}^{\infty}\tau(t)\phi(t)dt$ \ with
$\tau\in L^{1}(\mathbb{R})$, then%
\[
\mathcal{F}(T)(\phi)=\int_{-\infty}^{\infty}\mathcal{F(}\tau(t))\phi(t)dt
\]
If $\operatorname{Re}z>0$ then as we have seen above
\[
n_{1}\mapsto\int_{\mathfrak{a}^{\ast}}\int_{N}a(nn_{1}g)^{i\nu-(1+z)\rho
}a(n)^{-i\nu-(1+z)\rho}\overline{u(\nu)(k(n))}dn\mu(\nu)d\nu
\]
defines an element of $L^{1}(N_{1}).$ So Fubini's theorem implies that
\ $\tau_{z}\in L^{1}(\mathbb{R})$ if $\operatorname{Re}z>0$. If
$\operatorname{Re}z=0$ then%
\[
\tau_{0}(t)=\int_{\ker\chi}\int_{\mathfrak{a}^{\ast}}\left\langle \pi_{i\nu
}(\exp tx_{o}ng)1,\pi_{i\nu}(f)1\right\rangle \mu(\nu)d\nu dn=\int_{\ker\chi
}\overline{f(\exp tx_{o}ng)}dn.
\]
Since $f$, in particular, is in $\mathcal{C}(G/K)$ the function $n\mapsto
\overline{f(ng)}$ on $N$ is in $L^{1}(N)$. Thus if $z=0$ then $\tau_{z}\in
L^{1}(\mathbb{R})$. Let for $s,t\in\mathbb{R},$ $n\in\ker\chi,$ $\chi_{s}(\exp
tx_{o}n)=e^{its}$ \ Then $\chi_{1}=\chi$ and $\chi_{s}$ is generic for
$s\neq0$. Note that if $\operatorname{Re}z>0$ and $s\neq0$ then
\[
\mathcal{F}(\tau_{z})(s)=\int_{\mathfrak{a}^{\ast}}J_{\chi_{s}^{-1},i\nu
-z\rho}(\pi_{i\nu-z\rho}(g)1)\overline{J_{\chi_{s}^{-1},i\nu-\bar{z}\rho}%
(\pi_{i\nu}(f)1)}\mu(\nu)d\nu.
\]
Also if $z=0$ and $s\neq0$ then
\[
\mathcal{F}(\tau_{0})(s)=\overline{\int_{N}\chi_{s}(n)^{-1}f(ng)dn}.
\]
Define for $s\neq0$
\[
\sigma_{z}(s)=\int_{\mathfrak{a}^{\ast}}J_{\chi_{s}^{-1},i\nu-z\rho}(\pi
_{i\nu-z\rho}(g)1)\overline{J_{\chi_{s}^{-1},i\nu-\bar{z}\rho}(\pi_{i\nu
}(f)1)}\mu(\nu)d\nu
\]
then Proposition \ref{Key-Estimate} implies that $\sigma_{z}$ is continuous in
$z$ for $s\neq0$ and $\operatorname{Re}z\geq0,$ Furthermore, if
$\operatorname{Re}z>0$ and $s\neq0,$ then $\sigma_{z}(s)=\mathcal{F}(\tau
_{z})(s)$. We therefore have if $\phi$ has support in $\mathbb{R}-\{0\},$
\[
\lim_{%
\begin{array}
[c]{c}%
\operatorname{Re}z>0\\
z\rightarrow0
\end{array}
}\mathcal{F(\tau}_{z})(\phi)=\mathcal{F(\tau}_{0})(\phi)=\int_{-\infty
}^{\infty}\int_{N}\chi_{s}(n)f((ng)dn\phi(s)ds.
\]
Also%
\[
\lim_{%
\begin{array}
[c]{c}%
\operatorname{Re}z>0\\
z\rightarrow0
\end{array}
}\int_{-\infty}^{\infty}\int_{\mathfrak{a}^{\ast}}J_{\chi_{s}^{-1},i\nu-z\rho
}(\pi_{i\nu-z\rho}(g)1)\overline{J_{\chi_{s}^{-1},i\nu-\bar{z}\rho}(\pi_{i\nu
}(f)1)}\mu(\nu)d\nu\phi(s)ds=\int_{-\infty}^{\infty}\sigma_{0}(s)\phi(s)ds
\]%
\[
\int_{-\infty}^{\infty}\int_{\mathfrak{a}^{\ast}}J_{\chi_{s}^{-1},i\nu}%
(\pi_{i\nu}(g)1)\overline{J_{\chi^{-1},i\nu}(\pi_{i\nu}(f)1)}\mu(\nu)d\nu
\phi(s)ds
\]
This implies the theorem.
\end{proof}

Using the same methods as in the proof of the above theorem on can prove:

\begin{theorem}
Assume that $\chi$ is generic. Let $f$ be defined as in Theorem
\ref{HCVersion} then%
\[
\int_{N}\chi(n)^{-1}f(ng)dg=\int_{\mathfrak{a}^{\ast}}J_{\chi,i\nu}(\pi_{i\nu
}(g)1)\overline{J_{\chi,i\nu}(u)}\alpha(\nu)\mu(\nu)d\nu.
\]

\end{theorem}

\begin{corollary}
\label{Density}If $\beta\in C_{c}^{\infty}(\mathfrak{a}^{\ast})$ then
\[
\left(  g\mapsto\int_{\mathfrak{a}^{\ast}}J_{\chi,i\nu}(\pi_{i\nu}%
(g)1)\beta(\nu)\mu(\nu)d\nu\right)  \in\mathcal{C}(N\backslash G/K,\chi).
\]

\end{corollary}

\begin{proof}
Let $\omega$ be the support of $\beta$. If $\nu\in\omega$ let $u_{\nu}$ be a
$K$--finite element of $C^{\infty}(M\backslash K)$ such that $J_{i\nu}(u_{\nu
})\neq0$. Let $U_{\nu}$ be an open set with compact closure in $\mathfrak{a}%
^{\ast}$ such that $J_{i\lambda}(u_{\nu})\neq0$ for $\lambda\in U_{\nu}$. Then
the covering $U_{\nu}$ of $\omega$ has a finite refinement $U_{\nu_{1}%
},...,U_{\nu_{m}}.$ Let $\alpha_{1},...,\alpha_{r}$ be a partition of unity
subordinate to this covering of $\omega$. Then for each $i$ there exists
$w_{i}\in C^{\infty}(M\backslash K)$ that is $K$--finite and $W_{i}$ an open
set containing the support of $\alpha_{i}$ such that $J_{\chi,i\lambda}%
(w_{i})\neq0$ for $\lambda\in W_{i}.$ Define
\[
\varphi_{i}(\lambda)=\frac{\alpha_{i}(\lambda)\beta(\lambda)}{\overline
{J_{\chi,i\lambda}(w_{i})}}.
\]
Then $\varphi_{i}\in C_{c}^{\infty}(\mathfrak{a}^{\ast})$ and
\[
\beta(\lambda)=\sum_{i=1}^{r}\overline{J_{\chi,i\lambda}(w_{i})}\varphi
_{i}(\lambda).
\]
Now apply the previous theorem.
\end{proof}

\section{The spherical Whittaker Inversion Theorem and Plancherel Theorem.}

Let if $\chi$ is a generic character of $N$ set $J_{\chi,\nu}$ equal to the
corresponding Jacquet integral. We will need the following

\begin{lemma}
Assume $\chi$ is generic. Let $\psi\in\mathcal{C}(N_{o}\backslash G;\chi) $
and let $\varphi\in C_{c}^{\infty}(N)$ is such that
\[
\int_{N_{o}}\chi(n)^{-1}\varphi(n)dn=1.
\]
Set $f(nak)=\varphi(n)\psi(ak)$ for $n\in N_{o},a\in A_{o},k\in K$. Then
$f\in\mathcal{C}(G)$ and if $u\in C^{\infty}(M\backslash K)$ then
\[
J_{\chi^{-1},i\nu}(\pi_{i\nu}(f)u)=\int_{N\backslash G}J_{\chi^{-1},i\nu}%
(\pi_{i\nu}(g)u)\psi(g)dg
\]

\end{lemma}

\begin{proof}
Appendix 2, Corollary \ref{ExtensionProp} proves that $f\in\mathcal{C}(G)$. We
calculate%
\[
\int_{N\backslash G}J_{\chi^{-1},i\nu}(\pi_{i\nu}(g)u)\psi(g)dg==\int_{A\times
K}a^{-2\rho}J_{\chi^{-1},i\nu}(\pi_{i\nu}(ak)u)\psi(ak)dadk
\]%
\[
=\int_{N}\chi(n)^{-1}\int_{A\times K}a^{-2\rho}J_{\chi^{-1},i\nu}(\pi_{i\nu
}(ak)u)\varphi(n)\psi(ak)dadkdn
\]%
\[
=\int_{N\times A\times K}a^{-2\rho}J_{\chi^{-1},i\nu}(\pi_{i\nu}%
(nak)u)\varphi(n)\psi(ak)dadkda=\int_{G}J_{\chi^{-1},i\nu}(\pi_{i\nu
}(g)u)f(g)dg.
\]

For each $\nu\in\mathfrak{a}^{\ast}$ $J_{\chi^{-1},i\nu}$ is a continuous
functional on $C^{\infty}(M\backslash K)$. Also since $J_{\chi,i\nu}$ is tame
in the sense of Theorem 15.2.5 in \cite{RRGII}, for each $\nu$ in
$\mathfrak{a}^{\ast}$ there exist $C$ and $d$ such that if $u\in C^{\infty
}(M\backslash K)$ then%
\[
\left\vert J_{\chi^{-1}.i\nu}(\pi_{i\nu}(g)1)\right\vert \leq C\left\vert
g\right\vert ^{-\frac{1}{2}}(1+\log\left\Vert g\right\Vert )^{d}.
\]
This implies that%
\[
\int_{G}J_{\chi^{-1},i\nu}(\pi_{i\nu}(g)u)f(g)dg=J_{\chi^{-1},i\nu}(\int
_{G}f(g)\pi_{i\nu}(g)dg)=J_{\chi^{-1},i\nu}(\pi_{i\nu}(f)u).
\]

\end{proof}

\begin{theorem}
Let $\psi\in\mathcal{C}(N\backslash G/K;\chi)$. Set
\[
W_{\chi}(\nu,\psi)=\int_{N\backslash G}\overline{J_{i\nu}(\pi_{i\nu}(g)1)}%
\psi(g)dg
\]
then%
\[
\psi(g)=\int W_{\chi}(\nu,\psi)J_{i\nu}(\pi_{i\nu}(g)1)\mu(\nu)d\nu.
\]

\end{theorem}

\begin{proof}
Let $f$ be as in the preceding lemma for $\psi$. We observe that
\[
\overline{J_{i\nu}(\pi_{i\nu}(g)1)}=J_{\chi^{-1},-i\nu}(\pi_{-i\nu}(g)1).
\]
Thus the preceding lemma implies that%
\[
W_{\chi}(\nu,\psi)=J_{\chi^{-1},-i\nu}(\pi_{i\nu}(f)1).
\]
The spherical Plancherel Theorem implies that%
\[
f(g)=\int_{\mathfrak{a}^{\ast}}\left\langle \pi_{i\nu}(L_{g^{-1}%
}f)1,1\right\rangle \mu(\nu)dv=\int_{\mathfrak{a}^{\ast}}\left\langle
\pi_{i\nu}(f)1,\pi_{i\nu}(g)1\right\rangle \mu(\nu)dv
\]
so
\[
\bar{f}(g)=\overline{f(g)}=\int_{\mathfrak{a}^{\ast}}\left\langle \pi_{i\nu
}(g)1,\pi_{i\nu}(f)1\right\rangle \mu(\nu)dv.
\]
Theorem \ref{KeyFormula} implies that%
\[
\bar{f}_{\chi^{-1}}(g)=\int_{\mathfrak{a}^{\ast}}J_{\chi^{-1},i\nu}(\pi_{i\nu
}(g)1)\overline{J_{\chi^{-1},i\nu}(\pi_{i\nu}(f)1)}\mu(\nu)dv,
\]
Thus%
\[
\psi(g)=f_{\chi}(g)=\int_{\mathfrak{a}^{\ast}}\overline{J_{\chi^{-1},i\nu}%
(\pi_{i\nu}(g)1)}J_{\chi^{-1},i\nu}(\pi_{i\nu}(f)1)\mu(\nu)dv.
\]
The first part of this proof implies that%
\[
J_{\chi^{-1},i\nu}(\pi_{i\nu}(f)1)=W_{\chi}(-\nu,\psi).
\]
Also,%
\[
\overline{J_{\chi^{-1},i\nu}(\pi_{i\nu}(g)1)}=J_{-i\nu}(\pi_{-i\nu}(g)1).
\]
Hence%
\[
\psi(g)=\int_{\mathfrak{a}^{\ast}}W_{\chi}(-\nu,\psi)J_{-i\nu}(\pi_{-\nu
}(g)1)\mu(\nu)dv.
\]
This proves the theorem since $\mu(\nu)=\mu(-\nu)$.
\end{proof}

\begin{corollary}
With the notation as in the previous theorem, if $f,h\in\mathcal{C}%
(N\backslash G/K;\chi)$ then%
\[
\int_{N\backslash G}f(g)\overline{h(g)}dg=\int_{\mathfrak{a}^{\ast}}W_{\chi
}(\nu,f)\overline{W_{\chi}(\nu,h)}\mu(\nu)d\nu.
\]

\end{corollary}

\begin{proof}
We calculate. The previous theorem implies that%
\[
\int_{N\backslash G}f(g)\overline{h(g)}dg=\int_{N\backslash G}\int
_{\mathfrak{a}^{\ast}}W_{\chi}(\nu,f)J_{\chi,i\nu}(\pi_{i\nu}(g)1)\mu(\nu
)d\nu\overline{h(g)}dg
\]%
\[
=\int_{\mathfrak{a}^{\ast}}W_{\chi}(\nu,f)\overline{\int_{N\backslash
G}h(g)\overline{J_{\chi,i\nu}(\pi_{i\nu}(g)1)}dg}\mu(\nu)d\nu=\int
_{\mathfrak{a}^{\ast}}W_{\chi}(\nu,f)\overline{W_{\chi}(\nu,h)}\mu(\nu)d\nu.
\]

\end{proof}

The above result implies that the map
\[
W_{\chi}:\mathcal{C}(N\backslash G/K;\chi)\rightarrow L^{2}(\mathfrak{a}%
^{\ast},\mu(\nu)d\nu)
\]
extends to a continuous inner product preserving map of $L^{2}(N\backslash
G/K,\chi)$ to $L^{2}(\mathfrak{a}^{\ast},\mu(\nu)d\nu)$.

\section{\label{TodaChapter}The non-periodic Toda Lattice}

The original non-periodic Toda Lattice is the Hamiltonian system with
Hamiltonian%
\[
H(p,q)=\frac{1}{2}\sum_{i=1}^{n}p_{i}^{2}+\sum_{i=1}^{n-1}c_{i}^{2}%
e^{2(q_{i}-q_{i+1})}%
\]
$c_{i}\in\mathbb{R}-\{0\}$. Using the quantization rules (here Planck's
constant is normalized) $p_{j}\rightarrow i\frac{\partial}{\partial q_{j}}$
and $f(q)\rightarrow m_{f(q)}$ with $m_{f(q)}$ the operator on, $C^{\infty
}(\mathbb{R}^{n})$, given by multiplication by $f(q)$. Thus the quantum
Hamiltonian is%
\[
\mathcal{H=}-\frac{1}{2}\Delta_{q}+\sum_{i=1}^{n-1}c_{i}^{2}e^{2(q_{i}%
-q_{i+1})}%
\]
where
\[
\Delta_{q}=\sum_{i=1}^{n}\frac{\partial^{2}}{\partial q_{j}^{2}}.
\]

Consider the algebra, $\mathcal{A}$, of linear differential operators on
$\mathbb{R}^{n}$ with coefficients in the algebra generated by $e^{(q_{i}%
-q_{i+1})},i=1,...,n-1$. We take as a domain for this algebra the \ space,
$\mathcal{T}$, of $f\in C^{\infty}(\mathbb{R}^{n})$ such that
\[
t_{m,d,x}(f)=\sup_{q\in\mathbb{R}^{n}}e^{\sum_{j=1}^{n-1}m_{i}(q_{i}-q_{i+1}%
)}(1+\left\Vert q\right\Vert )^{d}|xf(q)|<\infty
\]
with $m=(m_{1},...,m_{n}),m_{i}$ and $d\in\mathbb{Z}_{\geq0}$ and $x$ is a
constant coefficient differential operator on $\mathbb{R}^{n}$ and endowed
with the topology induced by these semi-norms. This space is invariant under
$\mathcal{A}$ and $\mathcal{A}$ acts continuously on it. In \cite{Goodman-W2}
Toda 1 Section 2 it was shown the centralizer of $\mathcal{H}$ in
$\mathcal{W}$ is an algebra generated over $\mathbb{C}$ by $n$ elements
$D_{1}=\sum_{i=1}^{n}\frac{\partial}{\partial q_{j}},D_{2}=\mathcal{H}%
,...,D_{n}$ with algebraically independent symbols, $\sigma(D_{1}%
),...,\sigma(D_{n})$ generators for the $S_{n}$ invariant constant coefficient
differential operators. A solution to the quantum Toda lattice is thus a
family $K_{\nu}(q)$, $\nu\in\left(  \mathbb{R}^{n}\right)  ^{\ast}$ such that
\[
\int_{\left(  \mathbb{R}^{n}\right)  ^{\ast}}|K_{\nu}(q)f(q)|dq<\infty
,f\in\mathcal{T}%
\]
and
\[
D_{j}K_{\nu}=\sigma_{i}(\nu)K_{\nu},j=1,...,n.
\]
One has the following inversion formula: There exists a non-negative function
$\gamma(\nu)$ on $\left(  \mathbb{R}^{n}\right)  ^{\ast}$ such that
$\left\vert \gamma(\nu)\right\vert \leq C(1+\left\Vert \nu\right\Vert )^{r}$
for $\nu\in\left(  \mathbb{R}^{n}\right)  ^{\ast}$ and such that if
$f\in\mathcal{T}$ \ and if%
\[
\mathcal{K}(f)(\nu)=\int_{\mathbb{R}^{n}}f(q)\overline{K_{\nu}(q)}dq
\]
then%
\[
f(q)=\int_{\left(  \mathbb{R}^{n}\right)  ^{\ast}}\mathcal{K}(f)(\nu)K_{\nu
}(q)\gamma(\nu)d\nu.
\]
and if $f_{1},f_{2}\in\mathcal{T}$ then%
\[
\int_{\mathbb{R}^{n}}f_{1}(x)\overline{f_{2}(x)}dx=\int_{\left(
\mathbb{R}^{n}\right)  ^{\ast}}\mathcal{K}(f_{1})(\nu)\overline{\mathcal{K}%
(f_{2})(\nu)}\gamma(\nu)d\nu
\]

We now return to the situation of the preceding sections. Let $G$ and the
notation be as in Section 2 so we assume that, in particular, $G\subset
GL(n,\mathbb{R})$ for some $n$. In particular, $\mathfrak{g=}Lie(G)$. Let $C$
be the Casimir operator corresponding to the invariant form
\[
B(X,Y)=\mathrm{tr}XY
\]
for $X,Y$ in $\mathfrak{g}$. That is, if $X_{1},...,X_{m}$ is a basis of
$\mathfrak{g}$ and $Y_{1},...,Y_{n}$ are defined by $B(X_{i},Y_{j}%
)=\delta_{ij}$ then%
\[
C=\sum X_{i}Y_{i}\in U(\mathfrak{g}).
\]
Let $\theta,N,A,K,\mathfrak{n},\mathfrak{a},\mathfrak{k}$,$\Phi^{+},\Delta$ be
as before. Then $\mathfrak{n=}\sum_{\alpha\in\Phi^{+}}\mathfrak{n}_{\alpha}$
and let $X_{\alpha,j}$,$j=1,...,m_{\alpha}$ be an orthonormal basis of
$\mathfrak{n}_{\alpha}$ relative to the inner product
\[
\left\langle X,Y\right\rangle =-B(X,\theta Y).
\]
We define the generalized quantum non-periodic Toda Lattices associated with
$G$ to be the operator on $C^{\infty}(A)$ given by%
\[
L_{c}=-\frac{\sum h_{i}^{2}}{2}+\sum_{\alpha\in\Delta}c_{\alpha}^{2}%
a^{2\alpha}%
\]
with $c_{\alpha}\in\mathbb{R}-\{0\}$. For $G=GL(n,\mathbb{R})$ then take $A$
to be the group of diagonal $n\times n$ matrices with positive coefficients
then and $N$ the group of upper triangular $n\times n$ matrices with ones on
the main diagonal. Then identifying $\mathfrak{a}$ with $\mathbb{R}^{n}$ via
the map%
\[
(x_{1},...,x_{n})\mapsto\left[
\begin{array}
[c]{cccc}%
x_{1} & 0 & \cdots & 0\\
0 & x_{2} & \cdots & 0\\
0 & 0 & \ddots & 0\\
0 & 0 & \cdots & x_{n}%
\end{array}
\right]
\]
we have $\Delta=\{\alpha_{1},...,\alpha_{n-1}\}$ with $\alpha_{i}%
(x)=x_{i}-x_{i-1}$. Thus%
\[
L_{c}=-\frac{1}{2}\sum_{i=1}^{n}\frac{\partial^{2}}{\partial x_{j}^{2}}%
+\sum_{i=1}^{n-1}c_{\alpha_{i}}^{2}e^{2(x_{i}-x_{i+1})}%
\]

Set $K_{\nu}(x)=a^{-\rho}J_{i\nu}(\pi_{\nu}(\exp x)1)$ for $x\in\mathfrak{a}.$
Let $(...,...)$ also denote the complex bilinear extension of the dual form,
$(...,...)$ of $B_{|\mathfrak{a}}$ to $\mathfrak{a}_{\mathbb{C}}^{\ast}$.

\begin{proposition}
Let $\chi$ be a generic character of $N$. Set for $\alpha\in\Delta$
\[
c_{\alpha}^{2}=-\sum_{j=1}^{m_{\alpha}}\left(  d\chi(X_{\alpha,j})\right)
^{2}>0
\]
Let If $\nu\in\mathfrak{a}^{\ast}$ then%
\[
L_{c}a^{-\rho}K_{i\nu}(a)=\frac{\left\Vert \nu\right\Vert ^{2}}{2}a^{-\rho
}K_{i\nu}(a).
\]

\end{proposition}

\begin{proof}
We have for $\mu\in\mathfrak{a}_{\mathbb{C}}^{\ast}$
\[
CJ_{\nu}(\pi_{\nu}(g)1)=J_{\nu}(\pi_{\nu}(g)d\pi_{\nu}(C)1)=((\nu,\nu
)-(\rho,\rho))J_{\nu}(\pi_{\nu}(g)1).\text{ }%
\]
The proposition now follows directly from the calculations in Appendix 3.
\end{proof}

We note that if $\mathcal{W}(\mathfrak{a})$ is as in Appendix 2 and if
$\phi\in\mathcal{W}(\mathfrak{a})$ then the function $a^{-\rho}\phi$ is in the
space $\mathcal{T}(\mathfrak{a})$ defined as follows: Define for $u\in
C^{\infty}(\mathfrak{a})$ the semi-norm%
\[
t_{d,m,x}(u)=\sup_{x\in\mathfrak{a}}e^{\sum_{\alpha\in\Delta}m_{\alpha}%
\alpha(h)}(1+\left\Vert h\right\Vert )^{d}\left\vert x\phi(h)\right\vert
\]
with $m=\{m_{\alpha}|\alpha\in\Delta\},m_{\alpha},d\in\mathbb{Z}_{\geq0}$ and
$x$ is a constant coefficient differential operator on $\mathfrak{a}$.
$\mathcal{T}(\mathfrak{a})$ is the space of all $u$ in $C^{\infty
}(\mathfrak{a})$ such that all of the $t_{d,m,x}(u)<\infty$ \ endowed with the
topology induced by these semi-norms. Then the map
\[
\mu\mapsto(h\mapsto e^{-\rho(h)}\mu(\exp h)
\]
defines a topological isomorphism of $\mathcal{W}(\mathfrak{a})$ onto
$\mathcal{T}(\mathfrak{a})$ with inverse $u\mapsto(a\mapsto e^{\rho}u(\log
a))$. The main results of the previous section can be stated in the following form.

\begin{theorem}
If $u\in\mathcal{T}(\mathfrak{a}),\nu\in\mathfrak{a}^{\ast}$ set%
\[
\mathcal{K}(u)(\nu)=\int_{\mathfrak{a}}u(h)\overline{K_{\nu}(h)}dh
\]
then%
\[
u(h)=\int_{\mathfrak{a}^{\ast}}K_{\nu}(h)\mathcal{K}(u)(\nu)\mu(\nu)d\nu.
\]
Furthermore, if $u,w\in\mathcal{T}(\mathfrak{a})$ then%
\[
\int_{\mathfrak{a}}u(h)\overline{w(h)}dh=\int_{\mathfrak{a}^{\ast}}%
\mathcal{K}(u)(\nu)\overline{\mathcal{K}(w)(\nu)}\mu(\nu)d\nu.
\]

\end{theorem}

\begin{proof}
We note that if $u\in\mathcal{T}(\mathfrak{a})$ then $u(h)=e^{-\rho(h)}%
\phi(\exp h)$ with $\phi\in\mathcal{W}=\mathcal{C}(N\backslash G/K;\chi
)_{|A\text{.}}$ Thus if $\psi(nak)=\chi(n)\phi(a)$ then
\[
\mathcal{K}(u)(\nu)=\int_{A}a^{-\rho}\phi(a)a^{-\rho}\overline{J_{i\nu}%
(\pi_{i\nu}(a)1)}da=\int_{A}a^{-2\rho}\phi(a)\overline{J_{i\nu}(\pi_{i\nu
}(a)1)}da
\]%
\[
=\int_{N\backslash G}\psi(g)\overline{J_{i\nu}(\pi_{i\nu}(g)1)}dg=\mathcal{W}%
_{\chi}(\nu,\psi).
\]
The theorem says that%
\[
\psi(g)=\int_{\mathfrak{a}^{\ast}}\mathcal{W}_{\chi}(\nu,\psi)J_{i\nu}%
(\pi_{i\nu}(g)1)\mu(\nu)d\nu.
\]
So%
\[
u(h)=e^{-\rho(h)}\psi(\exp(h))=e^{-\rho(h)}\int_{\mathfrak{a}^{\ast}%
}\mathcal{K}_{\nu}(u)(\nu)J_{i\nu}(\pi_{i\nu}(\exp h)1)\mu(\nu)d\nu
=\int_{\mathfrak{a}^{\ast}}\mathcal{K}_{\nu}(u)(\nu)K_{\nu}(h)\mu(\nu)d\nu.
\]
The above formulas lead to the second assertion of the theorem.
\end{proof}

The rest of this section involves recalling several results from
\cite{Goodman-W2} which are necessary to the proof of the integrability of the
generalized Toda Lattices. Since $\mathfrak{g}=\mathfrak{n}\oplus
\mathfrak{a}\oplus\mathfrak{k}$ the Poincar\'{e}-Birkhof-Witt Theorem implies
that%
\[
U(\mathfrak{g})=U(\mathfrak{n}\oplus\mathfrak{a)\oplus}U(\mathfrak{g}%
)\mathfrak{k}\text{.}%
\]
Let $p$ denote the projection of $U(\mathfrak{g})$ onto $U(\mathfrak{n}%
\oplus\mathfrak{a)}$ corresponding to this direct sum decomposition. If $x\in
U(\mathfrak{g})^{\mathfrak{k}}$ (the centralizer of $\mathfrak{k}$ in
$U(\mathfrak{g)}$) and if $y\in U(\mathfrak{g})$ then $y=p(y)+\sum u_{i}Y_{i}$
and $x=p(x)+\sum w_{i}Y_{i},$with $Y_{i}\in\mathfrak{k},u_{i},w_{i}\in
U(\mathfrak{g}).$ Thus%
\[
yx=p(y)x+\sum u_{i}Y_{i}x=p(y)x+\sum u_{i}xY_{i}%
\]%
\[
=p(y)p(x)+p(y)\sum w_{i}Y_{i}+\sum u_{i}xY_{i}.
\]
This $p(yx)=p(y)p(x)$. Consider the two sided ideal in $U(\mathfrak{n}%
\oplus\mathfrak{a}),\mathcal{I=}U(\mathfrak{n}\oplus\mathfrak{a}%
)[\mathfrak{n},\mathfrak{n}]$. Then
\[
U(\mathfrak{n}\oplus\mathfrak{a})/\mathcal{I\cong}U(\mathfrak{n/}%
[\mathfrak{n},\mathfrak{n}]\oplus\mathfrak{a}).
\]
Let $\mu:$ $U(\mathfrak{n}\oplus\mathfrak{a})\rightarrow U(\mathfrak{n/}%
[\mathfrak{n},\mathfrak{n}]\oplus\mathfrak{a})$ be the corresponding
surjection and $q:U(\mathfrak{g})^{\mathfrak{k}}\rightarrow U(\mathfrak{n/}%
[\mathfrak{n},\mathfrak{n}]\oplus\mathfrak{a})$ be the corresponding
homomorphism, that is $q=\mu\circ p$. If $h\in\mathfrak{a}$ then define
$\partial_{h}f(x)=\frac{d}{dt}f(x+th)_{\backslash t=0}$. If $\lambda
\in\mathfrak{a}^{\ast}$ define $m_{\lambda}f=e^{\lambda}f$. Then
$[\partial_{h},m_{\lambda}]=\lambda(h)m_{\lambda}$. Let $\mathcal{A}$ be the
algebra of operators on $C^{\infty}(\mathfrak{a})$ generated by $\partial_{h}$
and $m_{\alpha}$ for $h\in\mathfrak{a}$ and $\alpha\in\Delta$. We define
$\tau(h)=\partial_{h}$ and $\tau(x)=d\chi(x)m_{\alpha}$ if $x\in
\mathfrak{n}/[\mathfrak{n}.\mathfrak{n]}_{\alpha}$. Then $\tau\circ q$ defines
a homomorphism of $U(\mathfrak{n/}[\mathfrak{n},\mathfrak{n}]\oplus
\mathfrak{a})$ onto $\mathcal{A}$. In \cite{Goodman-W2} we proved

\begin{theorem}
The centralizer of $\tau\circ q(C)$ in $\mathcal{A}$ is $\tau\circ
q(U(\mathfrak{g})^{\mathfrak{k}})$.
\end{theorem}

This result implies that the centralizer of $L_{m}$ in $\mathcal{A}$ is an
algebra generated by $\dim A$ elements with algebraically independent symbols.
The results of \cite{Goodman-W2}\ imply that

\begin{theorem}
If $D\in\mathcal{A}$ and $[D,L_{c}]=0$ then $DK_{\nu}=\sigma(D,i\nu)K_{\nu}$.
\end{theorem}

Thus in particular all of the assertions for the quantum non-periodic Toda
lattice have been proved.

\section{An implication}

The purpose of this section is to prove the following result which is noted
indirectly in the case of real rank 1 in [GW1].

Let the notation be as in the previous section. Let $\mathfrak{g}_{o}$ be a
semi-simple Lie subalgebra of $\mathfrak{g}$ that is split over $\mathbb{R}$,
is invariant under $\theta$, contains $\mathfrak{a}$, its root system relative
$\mathfrak{a}$ contains $\Delta$ as a set of simple roots. Let $G_{o}$ be the
connected subgroup of $G$ corresponding to $G_{o}$ and let $G_{o}=N_{o}AK_{o}$
be an Iwasawa decomposition of $G_{o}$ with $K_{o}=G_{o}\cap K$ and
$N_{o}=G_{o}\cap N$. We can now define the objects $J,c(\nu),\pi_{\nu},$ etc
for $G^{o}.$ We will use a superscript for an object if it pertains to the
group $G$ or respectively $G_{o}.$ That is $J^{G}$ or $J^{G_{o}}$. Also to
indicate the dependence on $\chi$ we will write $J_{\chi,\nu}^{G}$.

If $\chi$ is a generic character of $N$ let $d\chi_{|\mathfrak{n}_{\alpha}%
}=i\xi_{\alpha},\alpha\in\Delta$, as in the previous sections. The result that
is the purpose of this section is

\begin{theorem}
Let $\chi$ be a generic character of $G$ and let $\eta$ be a character of
$N_{o}$ such that $d\eta_{|\left(  \mathfrak{n}_{o}\right)  _{\alpha}}%
=i\psi_{\alpha}$ with$\left\Vert \psi_{\alpha}\right\Vert =\left\Vert
\xi_{\alpha}\right\Vert $. Then%
\[
c^{G_{o}}(\nu)a^{-\rho^{G}}J_{\chi,\nu}^{G}(\pi_{\nu}^{G}(a)1)=c^{G}%
(\nu)a^{-\rho_{o}^{G}}J_{\eta,\nu}^{G_{o}}(\pi_{\nu}^{G_{o}}(a)1).
\]

\end{theorem}

\begin{proof}
Fix $\nu$. Define $\Phi(h)=e^{-\rho^{G}(h)}J_{\chi,\nu}^{G}(\pi_{\nu}^{G}(\exp
h)1)$ for $h\in\mathfrak{a}$. Then we have seen that if $z\in\mathcal{A}%
^{L_{c}}$ with
\[
c_{\alpha}^{2}=-\left\Vert \xi_{\alpha}\right\Vert ^{2}%
\]
then
\[
z\Phi(h)=\sigma(z)(\nu)\Phi(h).
\]
Define for $g=nak$ with $n\in N_{o},a\in A$ and $k\in K_{o}$
\[
u(nak)=\eta(n)a^{\rho^{G_{o}}}\Phi(\log h)\text{.}%
\]
If $x\in U(\mathfrak{g}_{o})^{K_{o}}$ then $xu=\nu(\gamma^{G_{o}}(x))u$ with
$\gamma^{G_{o}}$ Harish-Chandra's homomorphism of $U(\mathfrak{g}_{o})^{K_{o}%
}$ onto $U(\mathfrak{a})^{W}$. Also if $y\in U(\mathfrak{g}_{o})$ then there
exists $C,d$ we have $\left\Vert yu(g)\right\Vert \leq C\left\Vert
g\right\Vert ^{d}$. The Casselman-Wallach theorem\ ??? implies that $u(g)$ is
a multiple of $J_{\nu}^{G_{o}}(g)$. This implies that there exists $\beta
(\nu)$ meromorphic such that
\[
a^{-\rho^{G}}J_{\chi,\nu}^{G}(\pi_{\nu}^{G}(a)1)=\beta(\nu)a^{-\rho^{G_{o}}%
}J_{\chi,\nu}^{G_{o}}(\pi_{\nu}^{G_{o}}(a)1).
\]
If $(\nu,\alpha)<0$ for all $\alpha\in\Delta$ then the limit formula in Lemma
\ref{asymptotic} implies that%
\[
c^{G}(\nu)=\beta(\nu)c^{G_{o}}(\nu).
\]
This implies the Lemma.
\end{proof}

\section{Appendix 1: Some inequalities}

The purpose of this appendix is to prove some estimates that will be used in
the body of the paper. The notation is as in Section 2.

\begin{lemma}
\label{Schwartz-estimate}Let $f\in\mathcal{C}(G/K)$ then for each $l\geq0$ and
$D$ a constant coefficient differential operator on $\mathfrak{a}^{\ast} $
there exists $B_{D,l}$ such that%
\[
\left\vert D(\pi_{i\nu}(f)1)(k)\right\vert \leq B_{D,l}(1+\left\Vert
\nu\right\Vert )^{-d},\nu\in\mathfrak{a}^{\ast},k\in K.
\]

\end{lemma}

\begin{proof}
By definition%
\[
\left(  \pi_{i\nu}(f)1\right)  (k)=\int_{G}f(g)a(kg)^{i\nu-\rho}dg=\int
_{G}f(k^{-1}g)a(g)^{i\nu-\rho}dg.
\]
Up to normalization of measures one has the standard integration formula (c.f.
\cite{HArmHom} 7.7.4)%
\[
\int_{G}\varphi(g)dg=\int_{\bar{N}\times A\times K}a^{2\rho}\varphi(\bar
{n}ak)d\bar{n}dadk.
\]
Thus%
\[
\left(  \pi_{i\nu}(f)1\right)  (k)=\int_{\bar{N}\times A}a^{2\rho}f(k^{-1}%
\bar{n}a)a^{i\nu-\rho}d\bar{n}da=\int_{\bar{N}\times A}a^{\rho}f(k^{-1}\bar
{n}a)a^{i\nu}d\bar{n}da.
\]
The map $\varphi\mapsto(h\mapsto e^{\rho(h)}\int_{\bar{N}}\varphi(\bar{n}\exp
h)d\bar{n}=\varphi^{\bar{P}})$ is a continuous map of $\mathcal{C}(G/K)$ to
$\mathcal{S}(\mathfrak{a})$ (c.f. Theorem 7.2.1 \cite{RRGI}.$.$The map
$f\mapsto L_{k}f$ is a continuous map of $K$ to $\mathcal{C}(G/K)$ and the
Fourier transform, $\mathcal{F}$ , is a continuous map from $\mathcal{S}%
(\mathfrak{a})$ to $\mathcal{S}(\mathfrak{a}^{\ast})$. So if $q$ is a
continuous semi-norm on $\mathcal{S}(\mathfrak{a}^{\ast})$ then $q(\mathcal{F}%
((R_{k}f)^{\bar{P}})\leq B_{q}.$ Thus if $q_{D,l}(u)=\sup_{\nu\in
\mathfrak{a}^{\ast}}(1+\left\Vert \nu\right\Vert )^{l}\left\vert D\alpha
(\nu)\right\vert $ for $l\geq0$ \ and $D$ a constant coefficient differential
operator on $\mathfrak{a}^{\ast}$ then since $K$ is compact
\[
\max q_{D,l}(\mathcal{F}((R_{k}f)^{\bar{P}}))\leq B_{D_{l},q_{l}}.
\]
This implies that
\[
\left\vert D\mathcal{F}((R_{k}f)^{\bar{P}})(\nu)\right\vert \leq B_{D,q_{l}%
}(1+\left\Vert \nu\right\Vert )^{-l}.
\]
Unraveling the above we have%
\[
\left(  \pi_{i\nu}(f)1\right)  (k)=\mathcal{F}((R_{k}f)^{\bar{P}})(-\nu)
\]
so the lemma is proved.
\end{proof}

\begin{lemma}
\label{simple-ineq}If $x,y\in G$ then $a(xg)^{-\rho}\leq\left\vert
g\right\vert ^{\frac{1}{2}}a(x)^{-\rho}.$
\end{lemma}

\begin{proof}
Let (as in section 2) $E=\wedge^{m}\mathfrak{g}$ ($m=\dim\mathfrak{n}$) and
let $u_{o}\in\wedge^{m}\mathfrak{\bar{n}}$ be a unit vector. If $g\in G$ with
$g=\bar{n}a(g)k$ with $\bar{n}\in\overline{N}$ and $k\in K$ then $\left\Vert
\wedge^{m}g^{-1}u_{o}\right\Vert =\left\Vert \wedge^{m}k^{-1}\wedge
^{m}a(g)^{-1}\wedge^{m}\bar{n}^{-1}u_{o}\right\Vert =a(g)^{2\rho}$. This
implies that if $x,g\in G$ then%
\[
a(x)^{2\rho}=\left\Vert \wedge^{m}g\wedge^{m}g^{-1}\wedge^{m}x^{-1}%
v_{o}\right\Vert \leq\left\vert g\right\vert a(xg)^{2\rho}%
\]
hence%
\[
a(xg)^{-\rho}\leq\left\vert g\right\vert ^{\frac{1}{2}}a(x)^{-\rho}.
\]

\end{proof}

Let $x_{o}\in\mathfrak{n}$ be orthogonal to $[\mathfrak{n},\mathfrak{n}]$ and
such that $d\chi(x_{o})=i$.

\begin{proposition}
\label{less-simple-ineq}Let $f\in\mathcal{C}(G/K)$ and if $\nu\in
\mathfrak{a}^{\ast}$ then set $u(\nu)=\pi_{i\nu}(f)1\in C^{\infty}(M\backslash
K).$ If $\omega$ is a compact subset of $G$ then there exist constants
$M,D,C_{\omega},d<\infty$ such that if $g\in\omega$ then
\[
\left\vert \int_{\ker\chi}\int_{\mathfrak{a}^{\ast}}\int_{N}1_{i\nu-(1+z)\rho
}(n_{1}\exp tx_{o})ng)\overline{u(\nu)_{i\nu-(1+\bar{z}\rho)}(n_{1})}dn_{1}%
\mu(\nu)d\nu dn\right\vert \leq C_{\omega}^{1+\operatorname{Re}z}%
DM(1+\left\vert t\right\vert )^{d}.
\]

\end{proposition}

\begin{proof}
We put the absolute values inside the integration. We are estimating%
\[
I=\int_{\ker\chi}\int_{\mathfrak{a}^{\ast}}\int_{N}a(n_{1}n(\exp
tx_{o})g)^{-(1+\operatorname{Re}z)\rho}a(n_{1})^{-(1+\operatorname{Re}z)\rho
}\left\vert u(\nu)(k(n_{1}))\right\vert dn_{1}\mu(\nu)d\nu dn.
\]
We note that
\[
\mu(\nu)\leq B(1+\left\Vert \nu\right\Vert )^{r}%
\]
for some $r$ and all $\nu\in\mathfrak{a}^{\ast}$. In Lemma
\ref{Schwartz-estimate} we showed that there exists a constant $L$ such that%
\[
\left\vert u(\nu)\right\vert \leq L_{m}(1+\left\Vert \nu\right\Vert )^{-r-m}%
\]
with $m$ arbitrary we take $m$ to be any $m>\dim\mathfrak{a}.$ Thus%
\[
I\leq M\int_{\ker\chi}\int_{N}a(n_{1}n(\exp tx_{o})g)^{-(1+\operatorname{Re}%
z)\rho}a(n_{1})^{-(1+\operatorname{Re}z)\rho}dn_{1}\mu dn
\]
with
\[
M=BL_{m}\int_{\mathfrak{a}^{\ast}}(1+\left\Vert \nu\right\Vert )^{-m}%
d\nu<\infty.
\]
Also, Lemma \ref{simple-ineq} implies that
\[
a(n_{1}n(\exp tx_{o})g)^{-(1+\operatorname{Re}z)\rho}\leq\left\vert
g\right\vert ^{\frac{^{1+\operatorname{Re}z}}{2}}a(n_{1}n(\exp tx_{o}%
))^{-(1+\operatorname{Re}z)\rho}\leq
\]%
\[
\left\vert g\right\vert ^{\frac{^{1+\operatorname{Re}z}}{2}}a(n_{1}n\exp
tx_{o}))^{-\rho}\leq\exp tx_{o}|^{\frac{1}{2}}\left\vert g\right\vert
^{\frac{^{1+\operatorname{Re}z}}{2}}a(n_{1}n)^{-\rho}.
\]
Since $\wedge^{m}Ad(extx_{o})$ is a polynomial in $t$ there exists a constant
$Q<\infty$ such that $\leq Q(1+\left\vert t\right\vert )^{d}$. Setting
$C_{\omega}=\max_{g\in\omega}\left\vert g\right\vert ^{\frac{1}{2}}$%
\[
I\leq MC_{\omega}^{\frac{1+\operatorname{Re}z}{2}}Q(1+\left\vert t\right\vert
)^{d}\int_{\ker\chi}\int_{N}a(n_{1}n)^{-(1+\operatorname{Re}z)\rho}%
a(n_{1})^{-(1+\operatorname{Re}z)\rho}dn_{1}dn
\]%
\[
\leq MC_{\omega}^{\frac{1+\operatorname{Re}z}{2}}Q(1+\left\vert t\right\vert
)^{d}\int_{\ker\chi}\int_{N}a(n_{1}n)^{-\rho}a(n_{1})^{-\rho}dn_{1}dn
\]%
\[
=MC_{\omega}^{\frac{1+\operatorname{Re}z}{2}}\left(  Q(1+\left\vert
t\right\vert )^{d}\right)  ^{\frac{1+\operatorname{Re}z}{2}}\int_{\ker\chi}%
\Xi(n)dn.
\]
As in the proof of Lemma \ref{simple-ineq} we have $a(g)^{\rho}=\left\Vert
\wedge^{m}Ad(g)^{-1}v_{o}\right\Vert ^{\frac{1}{2}}\leq\left\vert
g^{-1}\right\vert ^{\frac{1}{2}}.$ Also,%
\[
\Xi(x)\leq N\left\vert x\right\vert ^{-\frac{1}{2}}(1+\log\left\vert
x\right\vert )^{s}%
\]
for some $s,N<\infty$ (c.f. Theorem 5.5.3 \cite{RRGI}). Hence
\[
\Xi(x)=\Xi(x^{-1})\leq N\left\vert x^{-1}\right\vert ^{-\frac{1}{2}}%
(1+\log\left\vert x^{-1}\right\vert )^{s}\leq N_{\varepsilon}\left\vert
x^{-1}\right\vert ^{-\frac{1}{2}+\varepsilon}\leq N_{\varepsilon
}a(x)^{-(1-2\varepsilon)}%
\]
for each $\varepsilon>0$. Theorem \ref{Raphael} says that if $\varepsilon$ is
sufficiently small%
\[
\int_{\text{ker}\chi}a(n)^{-(1-2\varepsilon)\rho}dn<\infty.
\]
Completing the proof.
\end{proof}

\section{Appendix 2: The restriction of $\mathcal{C}(N\backslash G/K;\chi)$ to
$A$}

The purpose of this appendix is to give a complete description of the
restriction in its title.

\begin{lemma}
\label{basic}If $m=(m_{1},...,m_{l}),m_{i},d\in\mathbb{Z}_{\geq0}$ then there
exists a continuous semi-norm, $q_{m,d}$,on $\mathcal{C}(N\backslash G;\chi)$
such that if $f\in\mathcal{C}(N\backslash G;\chi)$ then $|f(\exp hk)|\leq
q_{m,d}(f)e^{\rho(h)}e^{-\sum m_{i}\alpha_{i}(h)}(1+\left\Vert h\right\Vert
)^{d}$ for $h\in\mathfrak{a}$.
\end{lemma}

\begin{proof}
Let $F=\{i|m_{i}>0\}$ and let $x_{1},...,x_{n}$ be a basis of $\mathfrak{g}$.
If $X\in\mathfrak{g}$ and if $k\in K$ then we can write $Ad(k)X=\sum
a_{i}(k,X)x_{i}$. Note that there exists $C$ such that
\[
\left\vert a_{i}(k,X)\right\vert \leq C\left\Vert X\right\Vert
\]
for all $k\in K$. Now let $X_{i}$ be an element of the $\alpha_{i}$ root space
in $\mathfrak{n}_{o}$ such that $d\chi(X_{i})=z_{i}\neq0$. Then
\[
f(\exp(h)k)=z_{i}^{-1}L_{X_{i}}f(\exp(h)k)=z_{i}^{-1}\frac{d}{dt}_{|t=0}%
f(\exp(tX_{i})\exp(h)k)=
\]%
\[
z_{i}^{-1}\frac{d}{dt}_{|t=0}f(\exp(h)\exp(tAd(\exp(-h))X_{i})k)=
\]%
\[
z_{i}^{-1}\frac{d}{dt}_{|t=0}f(\exp(h)\exp(te^{-\alpha_{i}(h)}X_{i})k)=
\]%
\[
z_{i}^{-1}\frac{d}{dt}_{|t=0}f(\exp(h)\exp(te^{-\alpha_{i}(h)}Ad(k^{-1}%
)X_{i})k)=
\]%
\[
e^{-\alpha_{i}(h)}z_{i}^{-1}\sum a_{j}(k^{-1},X_{i})R_{X_{j}}f(\exp(h)k).
\]
Iterating this argument yields an expression
\[
f(\exp(h)k)=e^{-\sum_{i\in F}m_{i}\alpha_{i}(h)}Z(k)f(ak)
\]
with $Z$ a smooth function from $K$ to $L=U^{\sum_{i\in F}m_{i}}%
(\mathfrak{g})$ with $U^{j}(\mathfrak{g})$ the standard filtration. If we
choose a basis of $L$, $y_{1},...,y_{r}$ then we have
\[
Z(k)=\sum b_{i}(k)y_{i}%
\]
with $b_{i}$ continuous functions on $K.$ Let $C_{j}=\max_{k\in K}\left\vert
b_{j}(k)\right\vert $. We have
\[
\left\vert f(\exp(h)k)\right\vert \leq e^{-\sum_{i\in F}m_{i}\alpha_{i}%
(h)}\sum_{j}C_{j}\left\vert y_{j}f(\exp(h)k)\right\vert \leq
\]%
\[
e^{-\sum_{i\in F}m_{i}\alpha_{i}(h)}(1+\left\Vert h\right\Vert )^{-d}%
e^{\rho_{o}(h)}\sum_{j}C_{j}q_{d,y_{j}}(f).
\]

\end{proof}

\begin{lemma}
\label{extension}Let $\psi\in C^{\infty}(G)$ be expressed in the form%
\[
\psi(nak)=\sum_{i=1}^{r}\sum_{j=1}^{s}\phi_{i}(n)f_{ij}(a)\gamma_{j}(k)
\]
for $n\in N,a\in A,k\in K$ with $r,s<\infty$, $\phi_{i}\in C_{c}^{\infty
}(N),\gamma_{j}\in C^{\infty}(K)$ and $f_{ij}\in C^{\infty}(A)$ such that if
$m=(m_{1},...,m_{l})$ with $m_{i}\in\mathbb{Z}_{\geq0},d\in\mathbb{Z}_{\geq0}$
and $x\in U(\mathfrak{a})$ then there exists $C_{ij,m,d,x}$ such that
\[
\left\vert xf_{ij}(a)\right\vert \leq C_{ij,m,d,x}a^{\rho}a^{-c_{1}\alpha
_{1}-...-c_{l}\alpha_{l}}(1+\left\Vert \log a\right\Vert )^{-d}.
\]
Then for $d$ $\in\mathbb{Z}_{\geq0}$ there exists $B_{d}$ such that
\[
\left\vert \psi(g)\right\vert \leq B_{d}\left\vert g\right\vert ^{-\frac{1}%
{2}}(1+\log\left\Vert g\right\Vert )^{-d}.
\]
Also, if $x,y\in U(\mathfrak{g)}$ then $L_{x}R_{y}\psi$ is of the same form.
\end{lemma}

\begin{proof}
To prove the inequality we may assume that $r,s=1$ so
\[
\psi(nak)=\phi(n)f(a)\gamma(k)
\]
Let $\omega$ be the the supports of $\phi$. Let $c_{1}\geq1$ be such that
\[
\max\{\left\Vert n\right\Vert ,\left\Vert n^{-1}\right\Vert \}\leq c_{1}%
,\min\{\left\Vert n\right\Vert ,\left\Vert n^{-1}\right\Vert \}\geq c_{1}%
^{-1},n\in\omega.
\]
We note that $\left\vert n\right\vert ^{\frac{1}{2}}=\left\Vert n\right\Vert $
and $|k|=\left\Vert k\right\Vert =1.$ We have for $n\in N,a\in A.k\in K$
\[
\left\vert nak\right\vert =\left\vert na\right\vert \leq c_{1}\left\vert
a\right\vert
\]
and%
\[
\left\vert a\right\vert =\left\vert n^{-1}nak\right\vert \leq c_{1}\left\vert
nak\right\vert .
\]
By the same argument we have the same inequalities for $\left\Vert
...\right\Vert $. If $h\in\mathfrak{a}$ let $s\in W(A)$ be such that
$\alpha(sh)\geq0,\alpha\in\Phi^{+}$. Then $\left\vert a\right\vert ^{\frac
{1}{2}}=\exp(sh)^{\rho}=a^{s^{-1}\rho}$ so if $s_{o}$ is the element of $W(A)$
such that $s_{o}\Phi^{+}=-\Phi^{+}$ then
\[
\left\vert nak\right\vert ^{-\frac{1}{2}}\leq c_{1}\left\vert a\right\vert
^{-\frac{1}{2}}=c_{1}a^{s^{-1}s_{o}\rho}.
\]
Also note that $s^{-1}s_{o}\rho=\rho-\sum_{i=1}^{l}u_{i}\alpha_{i}$ with
$u_{i}\in\mathbb{Z}_{\geq0}$. We also leave it to the reader to check that
there exists $c_{2}>0$ such that $\left\Vert a\right\Vert \leq e^{c_{2}%
\left\Vert \log a\right\Vert }$ for $a\in A$. Thus $(1+\log\left\Vert
a\right\Vert )\geq c_{3}(1+\left\Vert \log a\right\Vert ).$ With these
observations in place we have%
\[
\left\vert \psi(nak)\right\vert \leq\left(  \sup_{n\in\omega,k\in K}\left\vert
\phi(n)\gamma(k)\right\vert \right)  \left\vert f(a)\right\vert =c_{4}f(a)\leq
c_{4}C_{1,1,u,d}a^{\rho-\sum u_{i}\alpha_{i}}(1+\left\Vert \log a\right\Vert
)^{-d}
\]%
\[
=c_{4}C_{1,1,u,d}\left\vert a\right\vert ^{-\frac{1}{2}}c_{3}^{-d}%
(1+\log\left\Vert a\right\Vert )^{d}.
\]
Thus if we take the maxima of the $C_{1,1,u,d}$ for the $s\in W(A)$ and
incorporate the constants that appear in the estimates at the beginning of the
proof we have%
\[
\psi(g)|\leq C_{d}\left\vert g\right\vert ^{-\frac{1}{2}}(1+\log\left\Vert
g\right\Vert )^{-d}
\]
as asserted.

To complete the proof of the lemma we now we consider the derivatives. It is
enough to show that $R_{X}\psi$ and $L_{X}\psi$ are of the same form for
$X\in\mathfrak{g}$ \ We start with $R_{X}$. Again it is enough to show that if
$s,t=1$ then $R_{X}\psi$ is of the form indicated in the statement of the
lemma. Let $X_{1},...,X_{n}$ be a basis of $\mathfrak{g}$ such that
$X_{1},...,X_{r}\in\mathfrak{n}$ with $[h,X_{i}]=\beta_{i}(h)X_{i}$
$h\in\mathfrak{a}$, $X_{r+1},...,X_{r+l}\in\mathfrak{a}$, $X_{r+l+1}%
,...,X_{n}\in Lie(K)$. Then
\[
Ad(k)X=\sum c_{i}(k,X)X_{i}.
\]
We have%
\[
R_{X}\psi(nak)=\frac{d}{dt_{|t=0}}\psi(nak\exp tX)=\frac{d}{dt_{|t=0}}%
\psi(na\exp tAd(k)Xk)
\]%
\[
=\sum_{i=1}^{n}c_{i}(k,X)\frac{d}{dt_{|t=0}}\psi(na\exp tX_{i}k)=\sum
_{i=1}^{r}c_{i}(k,X)a^{\beta_{i}}\left(  R_{X_{i}}\phi(n)\right)
f(a)\gamma(k)
\]%
\[
+\sum_{i=r+1}^{r+l}c_{i}(k,X)\phi(n)\left(  R_{X_{i}}f(a)\right)
\gamma(k)=+\sum_{i=r+1}^{r+l}c_{i}(k,X)\phi(n)f(a)\left(  L_{X_{i}}%
\gamma(k)\right)
\]
which is easily seen to be of the right form.

To handle the left derivative we consider a different basis $Y_{i}=X_{i}$,
$i=1,...,r+l,Y_{r+l+1},...,Y_{r+l+m}$ a basis of $Lie(M)$ and $Y_{r+l+m+i}%
=\theta X_{i},i=1,...,r$. Then%
\[
Ad(n^{-1})X=\sum d_{i}(n,X)Y_{i}
\]
so%
\[
L_{X}\psi(nak)=-\sum_{i=1}^{n}d_{i}(n,X)\frac{d}{dt_{|t=0}}\psi(n\exp
tY_{i}ak)=-\sum_{i=1}^{r}d_{i}(n,X)\left(  R_{i_{i}}\phi(n)\right)
f(a)\gamma(k)
\]%
\[
+\sum_{i=r+1}^{r+l}d_{i}(n,X)\phi(n)L_{Y_{i}}f(a)\gamma(k)+\sum_{i=r+l=1}%
^{r+l+m}d_{i}(n,X)\phi(n)f(a)L_{Y_{i}}\gamma(k)+
\]%
\[
-\sum_{i=r+l+m+1}^{n}d_{i}(n,X)\frac{d}{dt_{|t=0}}\psi(n\exp tY_{i}ak).
\]
All but the last term are of the right form so we will show that it is also.
Set $\mu=r+l+m$ then
\[
\exp tY_{\mu+i}a=a\exp(tAd(a)^{-1}Y_{\mu+i})=a\exp(ta^{\beta_{i}}Y_{\mu+i}).
\]
So we are \ looking at
\[
-\sum_{i=1}^{r}d_{\mu+i}(n,X)a^{\beta_{i}}\frac{d}{dt_{|t=0}}\psi(na\exp
tY_{\mu+i}k)
\]
Now $Y_{\mu+i}+X_{i}=Z_{i}\in Lie(K).$ Thus $Y_{m+i}=Z_{i}-X_{i}$. So%
\[
-\sum_{i=1}^{r}d_{\mu+i}(n,X)a^{\beta_{i}}\frac{d}{dt_{|t=0}}\psi(na\exp
tY_{\mu+i}k)=\sum_{i=1}^{r}d_{\mu+i}(n,X)a^{\beta_{i}}\frac{d}{dt_{|t=0}}%
\psi(na\exp tX_{i}k)+
\]%
\[
-\sum_{i=1}^{r}d_{i}(n,X)a^{\beta_{i}}\frac{d}{dt_{|t=0}}\psi(na\exp
tZ_{i}k)=\sum_{i=1}^{r}d_{\mu+i}(n,X)a^{2\beta_{i}}R_{X_{i}}\phi
(n)f(a)\gamma(k)
\]%
\[
-\sum_{i=1}^{r}d_{\mu+i}(n,X)\phi(n)f(a)L_{Z_{i}}\gamma(k).
\]
The result is, finally, proved.
\end{proof}

\begin{corollary}
\label{ExtensionProp}If $f\in\mathcal{C}(N\backslash G;\chi)$ is right $K$
finite then and if $\phi\in C_{c}^{\infty}(N)$ then the function on $G$
defined by
\[
\psi(nak)=\phi(n)f(a)
\]
is in $\mathcal{C}(G)$.
\end{corollary}

\begin{proof}
Let $V=\mathrm{span}_{\mathbb{C}}\{R_{k}f|k\in K\}$. Then $\dim V<\infty.$ Let
$v_{1},...,v_{m}$ be a basis if $V$ then $R_{k}f=\sum\gamma_{i}(k)v_{i}$.
Thus
\[
f(ak)=\sum v_{i}(a)\gamma_{i}(k).
\]
Since, $v_{i}\in\mathcal{C}(N\backslash G;\chi)$ we see that $R_{x}v_{i|A}$
satisfies the inequalities for all $x\in U(\mathfrak{a})$. The result is now a
direct consequence of the definition of $\mathcal{C}(G)$ and Lemma
\ref{extension} .
\end{proof}

\begin{theorem}
If $\psi\in\mathcal{C}(G)$ then $\psi_{\chi}(g)=\int_{N}\chi(n)^{-1}%
\psi(ng)dn$ defines an element of $\mathcal{C}(N\backslash G;\chi)$.
\end{theorem}

\begin{proof}
The Harish-Chandra spherical function $\Xi(g)=\left\langle \pi_{0}%
(g)1,1\right\rangle $ satisfies
\[
\Xi(g)\geq\left\vert g\right\vert ^{-\frac{1}{2}}.
\]
Thus since%
\[
\left\vert \psi(na)\right\vert \leq C_{d}\left\vert na\right\vert
^{-1/2}(1+\log\left\Vert na\right\Vert )^{-d}%
\]
for all $d\geq0$ we have
\[
\left\vert \psi_{\chi}(ak)\right\vert \leq C_{d}\int_{N}\left\vert
na\right\vert ^{-\frac{1}{2}}(1+\log\left\Vert an\right\Vert )^{-d}dn\leq
C_{d}\int_{N}\Xi(na)(1+\log\left\Vert na\right\Vert )^{-d}dn.
\]
Also $\left\Vert na\right\Vert \geq\left\Vert a\right\Vert $ and
\[
\left\Vert n\right\Vert =\left\Vert naa^{-1}\right\Vert \leq\left\Vert
a^{-1}\right\Vert \left\Vert an\right\Vert =\left\Vert a\right\Vert \left\Vert
an\right\Vert \leq\left\Vert an\right\Vert ^{2}.
\]
Thus
\[
\left\vert \psi_{\chi}(ak)\right\vert \leq C_{d+r}\int_{N}\Xi(na)(1+\log
\left\Vert na\right\Vert )^{-d}dn(1+\log\left\Vert a\right\Vert )^{-r}.
\]
Now in the proof of Theorem 7.2.1 in \cite{RRGI} we have seen that there
exists $d$ such that%
\[
a^{-\rho}\int_{N}\Xi(na)(1+\log\left\Vert na\right\Vert )^{-d}dn\leq B<\infty
\]
Since $\left(  R_{x}\psi\right)  _{\chi}=R_{x}(\psi_{\chi})$ the theorem now
follow from the definition of $\mathcal{C}(N\backslash G;\chi)$.
\end{proof}

If $f\in C^{\infty}(\mathfrak{a})$ define for $m=(m_{1},...,m_{l}),d,$
$m_{i},d\in\mathbb{Z}_{\geq0}$ and $x$ a constant coefficient differential
operator on $\mathfrak{a}$
\[
w_{m,d}(f)=\sup_{h\in\mathfrak{a}}e^{-\rho(h)}e^{\sum m_{i}\alpha_{i}%
(h)}(1+\left\Vert h\right\Vert )^{d}\left\vert xf(h)\right\vert .
\]
Then we set $\mathcal{W(}\mathfrak{a})$ equal to the space of all $f\in
C^{\infty}(\mathfrak{a})$ such that $w_{c,d}(f)<\infty$ for all $m=(m_{1}%
,...,m_{l}),d,$ $m_{i},d\in\mathbb{Z}_{\geq0}$ endowed with the topology
determined by these semi-norms

\begin{theorem}
Assume that $\chi$ is generic. If $\psi\in C^{\infty}(N\backslash G/K;\chi)$
set $T(\psi)(h)=\psi(\exp h)$ for $h\in\mathfrak{a}$. Then $T(\mathcal{C}%
(N\backslash G/K;\chi))\subset\mathcal{W(}\mathfrak{a})$ and $T$ defines a
continuous isomorphism of $\mathcal{C}(N\backslash G;\chi)$ onto
$\mathcal{W(}\mathfrak{a})$.
\end{theorem}

\begin{proof}
Lemma \ref{basic} implies that $T(\mathcal{C}(N\backslash G/K;\chi
))\subset\mathcal{W(}\mathfrak{a})$. Using the definitions of the semi-norms
defining the topologies of $\mathcal{C}(N\backslash G/K;\chi)$ and
$\mathcal{W(}\mathfrak{a})$ the continuity of $T$ follows. $T$ is injective
since $G=NAK$. Let $\phi\in C_{c}^{\infty}(N)$ be such that%
\[
\int_{N}\chi(n)^{-1}\phi(n)dn=1.
\]
If $f\in\mathcal{W(}\mathfrak{a})$ set $\psi(nak)=\phi(n)f(a).$ Lemma
\ref{extension} implies that $\psi\in\mathcal{C}(G/K)$. Then $\psi_{\chi}%
\in\mathcal{C}(N\backslash G/K;\chi)$ and $T(\psi_{\chi})=f$. Thus the map is
surjective. The open mapping theorem now implies that $T^{-1}$ is continuous.
\end{proof}

\section{Appendix 3. The Whittaker radial component of the Casimir operator}

Let $\chi$ be a unitary character of $N$. Let $\mathfrak{m}$ be the
centralizer of $\mathfrak{a}$ in $\mathfrak{k}$. In this appendix we will
calculate the differential operator on $A$ corresponding to the Casimir
operator of $G$ on $C^{\infty}(N\backslash G/K;\chi).$

Let $X_{\alpha,i},i=1,...,m_{\alpha}$ be a basis of $\mathfrak{n}_{\alpha}$
such that $B(X_{\alpha,i},\theta X_{\alpha,j})=-\delta_{ij}.$ Note, before we
start calculating, that this implies that%
\[
\lbrack X_{\alpha,i},\theta X_{\alpha,j}]=-\delta_{ij}h_{\alpha}+m_{ij}%
\]
with $m_{ij}\in\mathfrak{m}$ and $m_{ij}=0$ if $i=j$ .Let $h_{1},...,h_{l}$ be
an orthonormal basis of $\mathfrak{a}$ relative to $\left\langle
...,...\right\rangle .$ Then the Casimir operator of $G$ relative to $B$ is
\[
C=-\sum_{j=1}^{m_{\alpha}}(X_{\alpha,j}\theta X_{\alpha,j}+\theta X_{\alpha
,j}X_{\alpha,j})+C_{\mathfrak{m}}+\sum_{i=1}^{l}h_{i}^{2}%
\]
Where $C_{\mathfrak{m}}$ is the Casimir operator corresponding to
$B_{|\mathfrak{m}}$. Let $f\in C^{\infty}(N\backslash G/K,\chi)$ we wish to
calculate $R_{C}f(a)$ for $a\in A$. We observe that $X_{\alpha,i}+\theta
X_{\alpha,j}\in Lie(K)$ and%
\[
\left(  X_{\alpha,i}+\theta X_{\alpha,j}\right)  ^{2}=X_{\alpha,i}^{2}+\theta
X_{\alpha,i}^{2}+X_{\alpha,j}\theta X_{\alpha,j}+\theta X_{\alpha,j}%
X_{\alpha,j}.
\]
Thus
\[
R_{X_{\alpha,j}\theta X_{\alpha,j}+\theta X_{\alpha,j}X_{\alpha,j}%
}f=-R_{X_{\alpha,i}^{2}}f-R_{\theta X_{\alpha,i}^{2}}f.
\]
Also,
\[
R_{\theta X_{\alpha,j}}R_{\theta X_{\alpha,j}}f=-R_{\theta X_{\alpha,j}%
}R_{X_{\alpha,j}}f=-R_{h_{\alpha}}f+R_{X_{\alpha,j}}R_{\theta X_{\alpha,j}%
}f=-R_{h_{\alpha}}f-R_{X_{\alpha,j}}^{2}f,
\]%
\[
\left(  R_{X_{\alpha,j}^{2}}f\right)  (a)=a^{2\alpha}\left(  L_{X_{\alpha
,j}^{2}}f\right)  (a)=d\chi(X_{\alpha,j})^{2}a^{2\alpha}f(a)
\]
and
\[
R_{C_{\mathfrak{m}}}f=0.
\]
The upshot is that
\[
Cf(a)=2\sum_{\alpha\in\Phi^{+}}\sum_{j=1}^{m_{\alpha}}d\chi(X_{\alpha,j}%
)^{2}a^{2\alpha}f(a)-\sum_{\alpha\in\Phi^{+}}m_{\alpha}h_{\alpha}%
f(a)+\sum_{i=1}^{l}h_{i}^{2}f(a).
\]
If $\alpha\notin\Delta$ then $d\chi(X_{\alpha,i})=0$ and $\sum_{\alpha\in
\Phi^{+}}m_{\alpha}h_{\alpha}=2h_{\rho}$, so%
\[
Cf(a)=2\sum_{\alpha\in\Delta}\sum_{j=1}^{m_{\alpha}}d\chi(X_{\alpha,j}%
)^{2}a^{2\alpha}f(a)-2h_{\rho.}f(a)+\sum_{i=1}^{l}h_{i}^{2}f(a).
\]
Noting that $d\chi=i\xi_{\chi}$ with $\xi_{\chi}\in\mathfrak{n}^{\ast}$ thus
if we set $\xi_{\chi,\alpha}=\xi_{\chi|\mathfrak{n}_{\alpha}}$
\[
\sum_{j=1}^{m_{\alpha}}d\chi(X_{\alpha,j})^{2}=-\left\Vert \xi_{\chi,\alpha
}\right\Vert ^{2}.
\]
We also have%
\[
a^{\rho}\sum_{i=1}^{l}h_{i}^{2}a^{-\rho}=(\rho,\rho)-2h_{\rho}+\sum_{i=1}%
^{l}h_{i}^{2}.
\]
We have derived

\begin{lemma}
if $f\in C^{\infty}(N\backslash G/K;\chi)$ and $a\in A$ then
\[
\left(  C+(\rho,\rho)\right)  f(a)=a^{\rho}(\sum_{i=1}^{l}h_{i}^{2}%
-2\sum_{\alpha\in\Delta}\left\Vert \xi_{\chi,\alpha}\right\Vert ^{2}%
a^{2\alpha})a^{-\rho}f(a).
\]

\end{lemma}


\begin{thebibliography}{99999}                                                                                            %


\bibitem[A]{Anker}J-Ph Anker, The Spherical Fourier Transform of rapidly
decreasing functions--a simple proof of a characterization due to
Harish-Chandra, Helgason, Trombi and Varadarajan, J. Func. Anal. 96 (1991), 331-349.

\bibitem[B]{raphael}Rapha\"{e}l Beauzart-Plessis, A local trace formula for
the Gan-Gross-Pasad conjecture for unitary groups: the Archimedian case,
Ast\'{e}risque, Livre-Tom 418, 2020.

\bibitem[GW1]{Goodman-W1}Roe Goodman and Nolan R. Wallach, Whittaker Vectors
and Conical Vectors,Journal of Functional Analysis 39, 199-279, 1980.

\bibitem[GW2]{Goodman-W2}Roe Goodman and Nolan R. Wallach,Classical and
Quantum-Mechanical Systems of Toda Lattice Type. I,2,3, Commun. Math. Phys. 83,94,105,1982,366-386,1984,177-217,1986,473-509.

\bibitem[H1]{Plancherel1}Harish-Chandra, Harmonic analysis on real reductive
groups 1 The theory of the constant term, J. Func. Anal., 19(1975),104-204.

\bibitem[H2]{Plancherel2}Harish-Chandra, Harmonic analysis on real reductive
groups II Wave-- Packets in the Schwartz Space, Inventiones math, 36 (1976), 241-318.

\bibitem[H3]{Plancherel3}Harish-Chandra, Harmonic analysis on real reductive
groups III The Maas Selberg relations and the Plancherel formula, Ann. of
Math. 104(1976),117-201.

\bibitem[H4]{CollectedI}Harish-Chandra, Collected Papers Volume I, Edited by
V.S.Varadarajan, Springer-Verlag, New York,1983.

\bibitem[H5]{CollectedV}Harish-Chandra, Collected Papers Volume V, Edited by
V.S.\ Varadarajan and Ramesh Gangolli, Springer, 2018.

\bibitem[H6]{Spher1}Harish-Chandra,\ Spherical Functions on a Semisimple Lie
Group, 1, \ Amer. J. Math. 80 (1958), 241-310

\bibitem[He1]{Helgason1}Sigurdur Helgason,\textit{Differential Geometry, Lie
Groups and Symmetric Spaces, }Academic Press, New York, 1978.

\bibitem[He2]{Hellgason2}Sigurdur Helgason, Geometric Analysis on Symmetric
Spaces, Mathematical Surveys and Monograpphs, Volume 39, Amarican Mathematical
Society, Providence, 1994.

\bibitem[K1]{Kostant1}Bertram Kostant, The solution to a generalized Toda
lattice and representation theory. Adv. in Math. 34 (1979), no. 3, 195--338.

\bibitem[K2]{KostantII}Bertram Kostant, On Whittaker vectors and
representation theory. Invent. Math. 48 (1978), no. 2, 101--184.

\bibitem[RRGI]{RRGI}Nolan R. Wallach, \textit{Real Reductive Groups I},
Academic Press, Boston, 1988.

\bibitem[RRGII]{RRGII}Nolan R. Wallach, \textit{Real Reductive Groups II},
Academic Press, Boston, 1992.

\bibitem[W1]{JacquetInt}Nolan R. Wallach, Lie Algebra Cohomology and
holomorphic continuation of generalized Jacquet integrals, Advanced Studies in
Pure Math.,14(1988).

\bibitem[W2]{HArmHom}Nolan R.\ Wallach, \textit{Harmonic Analysis on
Homogeneous Spaces, }Second Edition, Dover Publications, Inc., Mineola, 2018.

\bibitem[W3]{Whittaker-Plancherel}On the Whittaker Plancherel Theorem, To appear.
\end{thebibliography}
\end{document}